\documentclass[11pt]{article}
\usepackage{amsmath}
\usepackage{amsfonts}
\usepackage{amsthm}
\usepackage{amssymb, setspace}
\usepackage{color,graphicx,epsfig,geometry,hyperref,fancyhdr}
\usepackage[T1]{fontenc}
\usepackage{float}
\usepackage{subfig}
\usepackage{comment}
\usepackage{commath}
\usepackage{bbm}
\usepackage{listings}
\usepackage{mathrsfs,fleqn}
\usepackage{pgfplots}
\newtheorem{theorem}{Theorem}[section]
\newtheorem{proposition}{Proposition}[section]

\newtheorem{lemma}{Lemma}[section]

\newtheorem{remark}{Remark}[section]

\newcommand{\N}{\mathbb{N}}

\newcommand{\weakc}{\rightharpoonup}
\newcommand{\R}{\mathbb{R}}

\newcommand{\dnu}{\partial_\nu}
\newcommand{\grad}{\nabla}

\newcommand{\ep}{\varepsilon}
\newcommand{\domega}{{D \setminus \overline{\Omega}}}
\newcommand\dd{\,\mathrm{d}}

\numberwithin{equation}{section}

\begin{document}

\begin{flushleft}
\Large 
\noindent{\bf \Large Sampling methods for recovering buried corroded boundaries from partial electrostatic Cauchy data}
\end{flushleft}

\vspace{0.2in}

{\bf  \large Isaac Harris and Heejin Lee}\\
\indent {\small Department of Mathematics, Purdue University, West Lafayette, IN 47907 }\\
\indent {\small Email: \texttt{harri814@purdue.edu} and  \texttt{lee4485@purdue.edu} }\\

{\bf  \large Andreas Kleefeld}\\
\indent {\small Forschungszentrum J\"{u}lich GmbH, J\"{u}lich Supercomputing Centre, } \\
\indent {\small Wilhelm-Johnen-Stra{\ss}e, 52425 J\"{u}lich, Germany}\\
\indent {\small University of Applied Sciences Aachen, Faculty of Medical Engineering and } \\
\indent {\small Technomathematics, Heinrich-Mu\ss{}mann-Str. 1, 52428 J\"{u}lich, Germany}\\
\indent {\small Email: \texttt{a.kleefeld@fz-juelich.de}}\\


\begin{abstract}
\noindent We consider the inverse shape and parameter problem for detecting corrosion from partial boundary measurements. This problem models the non-destructive testing for a partially buried object from electrostatic measurements on the accessible part of the boundary. The main novelty is the extension of the linear sampling and factorization methods to an electrostatic problem with partial measurements. These methods so far have only mainly applied to recovering interior defects which is a simpler problem. Another important aspect of this paper is in our numerics, where we derive a system of boundary integral equations to recover the mixed Green's function which is needed for our inversion. With this, we are able to analytically and numerically solve the inverse shape problem. For the inverse parameter problem, we prove uniqueness and Lipschitz-stability (in a finite dimensional function space) assuming that one has the associated Neumann-to-Dirichlet operator on the accessible part of the boundary. 
\end{abstract}

\section{Introduction}
In this paper, we consider an inverse shape and parameter problem coming from electrical impedance tomography (EIT). The model we study is for a partial buried object that was degraded via corrosion. This problem is motivated by non-destructive testing where one wishes to detect/recover the corroded part of the boundary without removing the object. To this end, we will study the linear sampling and factorization methods for recovering the corroded boundary. These methods were first introduced in \cite{CK,firstFM}, respectively. This is novel due to the fact that we have data only on the accessible part of the boundary and we wish to recover the rest of the boundary. Our inversion is done by embedding the defective region into a `healthy' region and comparing the gap in voltages. The linear sampling and factorization methods have been studied for similar problems in \cite{EIT-granados1,EIT-granados2,GLSM-elastic,EIT-finiteElectrode,eit-harris,regfm1,applied-lsm} where one wishes to recover interior defects from either full or partial boundary data. Again, this problem is different in the fact that we have partial boundary data and we wish to recover unaccessible part of the boundary. 

In order to solve the inverse shape problem we will consider two well known {\it qualitative} reconstruction methods i.e. the linear sampling and factorization methods. These methods have been greatly studied over the years for different inverse shape problems, see \cite{fm-waveguide,FM-wave,fm-gbc,fm-GR,FMiso,glsm-stokes,lsm-heat,LSM-periodic} as well as the manuscripts \cite{Colto2013,kirschbook} and the references therein. Iterative methods for this problem were studied in \cite{CK-ibc,CK-ibc2} which extends the method presented in \cite{Rundell}. In the aforementioned papers, a non-linear system of boundary integral equations was used to solve the inverse shape and parameter problems. We also note that in \cite{CK-gibc} the authors used a similar iterative method to solve the problem with a generalized impedance condition. One of the main advantageous for using a qualitative method is the fact that one needs little {\it a priori} information about the region of interest.  On the other hand iterative methods will often converge to a local minima rather than the global minima if the initial guess is not sufficiently close to the target. Therefore, in many non-destructive testing applications it may be useful to use a qualitative method. 

We also consider the inverse parameter problem. Here, we will assume that the corroded part of the boundary is known/recovered. Then, we prove that the knowledge of the Neumann-to-Dirichlet operator on the accessible part of the boundary can uniquely recover the corrosion(i.e. Robin) parameter. Once we have proven uniqueness, we then turn our attention to stability. Due to the fact that inverse EIT problem are exponentially ill-posed there is no hope to obtain a Lipschitz--stability estimate on standard function spaces. Here, we appeal to the techniques in \cite{EIT-InnerStability,eit-transmission2,invrobin-Meftahi} to prove Lipschitz--stability assuming the parameter is in a finite dimensional function space. This is useful for numerical reconstructions of the parameter since one will often discretize the unknown function to be a linear combination of finite basis functions. Numerical reconstructions of the parameter are not studied here but the algorithm in \cite{eit-transmission1} can also be applied to this problem.

The rest of the paper is structured as follows. In Section \ref{sect-dp} we setup the direct and inverse problem under consideration. Then, in Section \ref{sect-ip1} we consider the inverse shape problem where we give the theoretical justification of the linear sampling and factorization methods for our model. This will give a computationally simple yet analytically rigorous method for recovering the corroded part of the boundary. Next, in Section \ref{sect-ip2} we consider the inverse parameter problem assuming that the corroded boundary is known/recovered, where we prove uniqueness and stability with respect to the Neumann-to-Dirichlet operator. Then, we provided numerical examples in Section \ref{sect-numerics} for recovering the corroded boundary. Finally, a summary and conclusion is given.

\section{The Direct Problem}\label{sect-dp}
In this section, we will discuss the direct problem associated with the inverse problems under consideration. Again, this problem comes from EIT where one applies a current on the accessible part of the boundary and measures the resulting voltage. To begin, we let the known region $D \subset \mathbb{R}^m$ for $m=2, 3$ be an open bounded and simply connected domain with the piecewise $C^1$ boundary $\partial D$. The boundary $\partial D$ can be decomposed into 
$$\partial D =  \overline{\Gamma_N} \cup \overline{\Gamma_D} \quad \text{where} \quad \Gamma_N \cap \Gamma_D = \emptyset $$
and are relatively open subsets of $\partial D$. We assume that part of the region $D$ has been buried such that $\Gamma_N$ is the accessible part of the boundary with $\Gamma_D$ being the part of the boundary that has been buried. Being buried has caused part of the region to be corroded away. The part of the region that has corroded away will be denoted $\Omega$. To this end, we let $\Omega \subset D$ be an open subset of $D$ such that a part of the boundary $\partial \Omega$ is $\Gamma_D$ and the other part of the boundary is $C^1$ and denoted by $\Gamma_C$. Therefore, we have that 
$$\partial (D \setminus \overline{\Omega}) = \overline{\Gamma_N} \cup \overline{\Gamma_C} \quad \text{where} \quad \Gamma_N \cap \Gamma_C = \emptyset $$
and are relatively open. In Figure \ref{domain}, we have illustrated the aforementioned setup. 
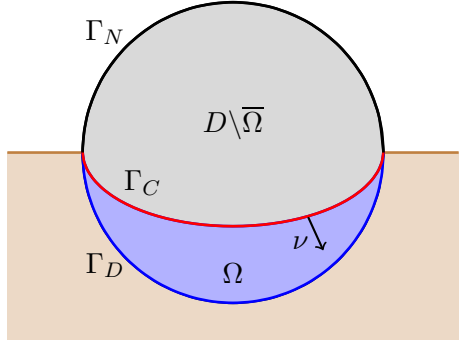
\begin{figure}[H]
    \centering
    \begin{tikzpicture}
       \draw [fill=brown!30, draw=none] (-3,-2.5) rectangle (3,0);
			\draw[line width=1pt, brown] (-3,0) -- (3,0);
        \draw[line width=1pt, smooth, black, fill=gray!30, scale=2] plot[domain=0:2*pi] ({cos(\x r)}, {sin(\x r)});
        \draw[line width=1pt, smooth, blue, fill=blue!30, scale=2] plot[domain=pi:2*pi] ({cos(\x r)}, {sin(\x r)});
        \draw[line width=1pt, smooth, blue, fill=gray!30, scale=2] plot[domain=pi:2*pi] ({cos(\x r)}, {0.5*sin(\x r)+0.01});
        \draw[line width=1pt, smooth, red, fill=none, scale=2] plot[domain=pi:2*pi] ({cos(\x r)}, {0.5*sin(\x r)+0.01});
        \draw[line width=1pt, smooth, black, fill=none, scale=2] plot[domain=0:pi] ({cos(\x r)}, {sin(\x r)});
     
      \draw[->,thick] (1,-0.86) -- (1.2,-1.3);
        \node at (0.9,-1.2) {$\nu$};
        
        \node at (0,0.4) {$D\backslash \overline{\Omega}$};
        \node at (0,-1.6) {$\Omega$}; 
        \node at (-1.7,-1.5) {$\Gamma_D$}; 
        \node at (-1.7,1.6) {$\Gamma_N$};  
        \node at (-1.2,-0.4) {$\Gamma_C$}; 
    \end{tikzpicture}
    \caption{The partially buried domain $D$ with the corroded part $\Omega$.}
    \label{domain}
\end{figure}

In order to determine if there is a non-trivial corroded region $\Omega$ we assume that a current denoted $g$ is applied to the accessible part of the boundary $\Gamma_N$. This will produce an electrostatic potential function $u$ for the defective material $\domega$.  This gives that the direct problem can be modeled by the mixed Neumann--Robin boundary value problem: given $g \in L^2(\Gamma_N)$, determine $u \in H^{1}(\domega)$ such that
\begin{align}\label{dp1}
\Delta u= 0 \quad  \text{in } D\setminus \overline{\Omega}, \quad  \dnu{u} = g \quad  \text{on } \Gamma_N, \, \text{ and } \, \dnu{u} +\gamma u  = 0 \quad  \text{on }\Gamma_C.
\end{align}
Here $\nu$ denotes the outward unit normal to $\domega$ and the corrosion coefficient $\gamma \in L^\infty(\Gamma_C)$. We will assume that there are two real--valued constants $\gamma_{\text{max}}$ and $\gamma_{\text{min}}$ such that the corrosion coefficient satisfies 
$$0< \gamma_{\text{min}} \leq  \gamma (x) \leq \gamma_{\text{max}}  \quad \text{for a.e.} \quad x \in \Gamma_C.$$ 
Note that our notation is that $\Gamma_N$ is the Neumann boundary and that $\Gamma_C$ corresponds to the corroded/Robin boundary.

Now, we wish to establish the well-posedness of the direct problem \eqref{dp1}. This can be done by considering the equivalent variational formulation of \eqref{dp1}. To this end, we can take $\varphi \in H^{1}(\domega)$ and using Green's first identity we have that 
\begin{align}\label{uvarf}
\int_{\domega} \nabla u \cdot \nabla {\varphi} \dd x + \int_{\Gamma_C} \gamma u {\varphi} \dd s
=
\int_{\Gamma_N} g {\varphi} \dd s.
\end{align}
Note that \eqref{uvarf} is satisfied for any test function $\varphi \in H^{1}(\domega)$. Clearly this implies that the variational formulation is given by 
$${A}(u,\varphi) = \ell (\varphi)$$ 
where the sesquilinear form ${A}(\cdot , \cdot): H^1(\domega) \times H^1(\domega) \rightarrow \mathbb{R}$ is given by
\begin{align*}
{A}(u, \varphi) = 
\int_{\domega} \nabla u \cdot \nabla {\varphi} \dd x + \int_{\Gamma_C} \gamma u {\varphi} \dd s
\end{align*}
and
the conjugate linear functional $\ell: H^1(\domega) \rightarrow \mathbb{R}$ is given by
$$\ell(\varphi) = \int_{\Gamma_N} g {\varphi} \dd s.$$
Here, the integrals over $\Gamma_C$ and $\Gamma_N$ are interpreted as the inner--product on $L^2(\Gamma_C)$ and $L^2(\Gamma_N)$, respectively. These integrals are well defined by the Trace Theorem. For the well-posedness, notice that 
$${A}(u, u) \geq \|\grad u \|^2_{L^2(\domega)} +\gamma_{\text{min}} \|u\|^2_{L^2(\Gamma_C)}$$
which implies that $A(\cdot , \cdot)$ is coercive on $H^1(\domega)$ by appealing to a standard Poincar\'{e} type argument (see for e.g. \cite{salsa} p. 487). From the Trace Theorem, we have that 
\begin{align*}
|\ell(\varphi)| \leq C\|g\|_{L^2(\Gamma_N)}\|\varphi\|_{H^1(\domega)}.
\end{align*}
With this we have the following result. 
\begin{theorem}\label{wpnu}
The mixed Neumann--Robin boundary value problem \eqref{dp1} has a unique solution $u \in H^1(\domega)$ that satisfies the estimate
\begin{align*}
\|u\|_{H^1(\domega)} \leq C \|g\|_{L^2(\Gamma_N)}
\end{align*}
with $C$ independent of $g \in L^2(\Gamma_N)$.
\end{theorem}

With the well-poseness of \eqref{dp1} established, we now consider an auxiliary boundary value problem for the electrostatic potential in the healthy domain $D$: given $g \in L^2(\Gamma_N)$, determine $u_0 \in H^1(D)$ such that 
\begin{align}\label{dpu0}
\Delta u_0= 0  \text{ in } D, \quad \dnu{u_0} = g  \text{ on } \Gamma_N, \,  \text{ and } \, u_0 = 0 \text{ on }\Gamma_D.
\end{align}
Similarly, it can be shown that the above boundary value problem \eqref{dpu0} is well-posed with the estimate 
$$\|u_0\|_{H^1(D)} \leq C \|g\|_{L^2(\Gamma_N)}.$$ 
This can be done by again appealing to the variational formulation as well as the fact that $u_0$ satisfies the Poincar\'{e} estimate 
$$\|u_0\|_{H^1(D)} \leq C \| \grad u_0 \|_{L^2(D)} $$
due to the zero trace on $\Gamma_D$ (see \cite{salsa} p. 486). Note that in our notation $\Gamma_D$ is the part of the boundary where we impose the homogeneous Dirichlet condition. Since $D$ is known a priori we have that $u_0$ can always be computed numerically. 

In order to determine the corroded subregion $\Omega$, we will assume that the $u|_{\Gamma_N}$ can be measured and that $u_0|_{\Gamma_N}$ can be computed for any current $g\in L^2(\Gamma_N)$.  Now we define the \textit{Neumann-to-Dirichlet} (NtD) operators 
\begin{align}\label{NtDop}
\Lambda \,  \textrm{ and } \, \Lambda_0 : L^2(\Gamma_N) \to L^2(\Gamma_N) \quad \text{given by} \quad  \Lambda g = u|_{\Gamma_N} \quad \text{and} \quad \Lambda_0 g = u_0|_{\Gamma_N},
\end{align}
where $u$ and $u_0$ are the unique solutions to \eqref{dp1} and \eqref{dpu0}, respectively. From the well-posedness of the boundary value problems it is clear that the operators $\Lambda$ and $\Lambda_0$ are well-defined bounded linear operators. The inverse problems that we are interested in are the \textit{inverse shape problem} of determining the corroded region $\Omega$ from the knowledge of the difference $(\Lambda-\Lambda_0)$ and the \textit{inverse impedance problem} of recovering the corrosion coefficient $\gamma$ provided that $\Gamma_C$ is known. In the following section, we will study the linear sampling and factorization methods to recover $\Omega$. Then, we will turn our attention to proving that $\Lambda$ uniquely recovers the corrosion coefficient $\gamma$ as well as provide a stability estimate.

\section{The Inverse Shape Problem}\label{sect-ip1}
In this section, we are interested in the inverse shape problem of recovering $\Omega$ from the knowledge of the NtD operators. In order to solve this problem we will consider the linear sampling and factorization methods associated with $(\Lambda-\Lambda_0)$. Our analysis will show that the linear sampling method can give an approximate reconstruction of $\Omega$ where as the factorization method can be used under a stricter set of assumptions on the corrosion coefficient. The factorization method is mathematically more advantageous to use due to the fact that it gives an explicit characterization of the region of interest from the spectral decomposition of an operator defined by the difference of the NtD maps. In either case, we need to decompose the operator $(\Lambda-\Lambda_0)$ to obtain a more explicit relationship with the unknown region $\Omega$.  

To begin, we first derive an initial factorization of the operator $(\Lambda-\Lambda_0)$. Therefore, we first notice that the difference of the electrostatic potentials satisfies 
$$\Delta (u-u_0)= 0  \text{ in } \domega, \quad \dnu{(u-u_0)} = 0  \text{ on } \Gamma_N, \,  \text{ and } \,(u- u_0)|_{\Gamma_C} \in H^{1/2}(\Gamma_C). $$
This motivates us to consider the auxiliary boundary value problem: given $\varphi \in H^{1/2}(\Gamma_C)$, determine  $w \in H^1(\domega)$ such that 
\begin{align}\label{dp2}
\Delta w= 0  \text{ in } \domega, \quad \dnu{w} = 0 \,  \text{ on } \, \Gamma_N, \,  \text{ and } \,w=\varphi \,  \text{ on } \, \Gamma_C.
\end{align}
Arguing similarly to the previous section, we have that \eqref{dp2} is well-posed which implies that we can define
\begin{align}\label{opg}
G: H^{1/2}(\Gamma_C) \to L^2(\Gamma_N) \quad \text{given by } \quad G\varphi=w|_{\Gamma_N},
\end{align}
where $w$ is the solution to \eqref{dp2} as a bounded linear operator by appealing to the Trace Theorem. By the well-posedness of \eqref{dp2}, we see that if 
$$ \varphi =  (u- u_0)|_{\Gamma_C} \quad \text{ we obtain that} \quad w = (u- u_0) \, \, \text{ in } \, \domega.$$
Now, we further define the bounded linear operators 
\begin{align}\label{opl}
L \, \text{ and } \, L_0: L^2(\Gamma_N) \rightarrow H^{1/2}(\Gamma_C) \quad \text{given by } \quad Lg=u|_{\Gamma_C} \, \text{ and } \, L_0 g=u_0|_{\Gamma_C}. 
\end{align}
With this we have our initial factorization $(\Lambda-\Lambda_0) = G(L-L_0)$. 

With our initial factorization in hand we will analyze the properties of the operators defined above. First, we notice that due to the compact embedding of $H^{1/2}(\Gamma_N)$ into $L^2(\Gamma_N)$ we have compactness of the operator $G$ defined in \eqref{opg}. We also notice that by Holmgren's theorem (see for e.g. \cite{holmgren}) if $\varphi$ is in the null-space of $G$, this would imply that 
$$w = \dnu{w} = 0 \,  \text{ on } \, \Gamma_N \quad \text{ giving that } \quad w = 0 \,  \text{ in } \, \domega.$$
By the Trace Theorem $\varphi = 0$ which gives injectivity of the operator $G$. With this we now present a result that gives the analytical properties of the source-to-trace operator $G$. 

\begin{theorem}\label{opG1}
The operator $G: H^{1/2}(\Gamma_C) \to L^2(\Gamma_N)$ as defined in \eqref{opg} is compact and injective.
\end{theorem}

With this, in order to further analyze the operator $G$ we need to compute its adjoint. The adjoint operator $G^*$ will be a mapping from $L^2(\Gamma_N)$ into $\widetilde{H}^{-1/2}(\Gamma_C)$. Note that the adjoint is computed via the relationship
$$\big(G \varphi , \psi \big)_{L^2(\Gamma_N)} = \big\langle \varphi , G^*\psi \big\rangle_{\Gamma_C}$$ 
where $\big\langle \cdot , \cdot \big\rangle_{\Gamma_C}$ is the sesquilinear dual--product between the 
$$\text{Hilbert Space $H^{\pm1/2}(\Gamma_C)$ and its dual space $\widetilde{H}^{\mp 1/2} (\Gamma_C )$}$$ 
where $L^2  (\Gamma_C)$ is the associated Hilbert pivot space, see \cite{McLean} p. 99 for details. The Sobolev space $\widetilde{H}^{s} (\Gamma_C )$ is the closure of $C^\infty_0(\Gamma_C)$ with respect to the ${H}^{s} (\Gamma_C )$--norm for any $s\in \R$. Now, with this in mind we can give another result for the analytical properties of $G$. 

\begin{theorem} \label{opG2}
The adjoint $G^{*}: L^2(\Gamma_N) \to \widetilde{H}^{-1/2}(\Gamma_C)$ is given by $G^{*}\psi =  -  \partial_{\nu} v |_{\Gamma_C}$ where $v \in H^1 (\domega)$ satisfies 
\begin{align} \label{dpv}
\Delta v= 0  \text{ in } \domega, \quad \dnu{v} = \psi \,  \text{ on } \, \Gamma_N, \,  \text{ and } \, v =0 \,  \text{ on } \, \Gamma_C.
\end{align} 
Moreover, the operator $G$ has a dense range.
\end{theorem}
\begin{proof}
To prove the claim, we first note that \eqref{dpv} is well-posed for any $\psi \in L^2(\Gamma_N)$. Now, in order to compute the adjoint operator we apply Green's second identity to obtain 
$$ 0 =  \int_{\Gamma_N} {v} \, \partial_{\nu}w - w \, \partial_{\nu} \overline{v} \, \text{d}s + \int_{\Gamma_C} {v} \, \partial_{\nu}w - w \, \partial_{\nu} {v} \, \text{d}s,$$
where we have used the fact that both $w$ and $v$ are harmonic in $\domega$ as well as the fact that $\partial (D \setminus \overline{\Omega}) = \overline{\Gamma_N} \cup \overline{\Gamma_C}$ with $\Gamma_N \cap \Gamma_C = \emptyset$. Using the boundary conditions in \eqref{dp2} and \eqref{dpv} we have that 
$$\int_{\Gamma_N} w \, {\psi} \, \text{d}s = -  \int_{\Gamma_C} \varphi \, \partial_{\nu} {v} \, \text{d}s.$$
Notice, that the left hand side of the above equality is a bounded linear functional of $\varphi \in H^{1/2}(\Gamma_C)$. Therefore, by definition we have that $G\varphi=w|_{\Gamma_N}$ which implies that 
$$\big(G \varphi , \psi \big)_{L^2(\Gamma_N)} =\int_{\Gamma_N} w \, {\psi} \, \text{d}s = -  \int_{\Gamma_C} \varphi \, \partial_{\nu} {v} \, \text{d}s = \big\langle \varphi , G^*\psi \big\rangle_{\Gamma_C}$$ 
proving that $G^{*}\psi =  -  \partial_{\nu} v |_{\Gamma_C}$. 

Now, proving that the operator $G$ has a dense range is equivalent to proving that the adjoint $G^*$ is injective (see \cite{Brezis}, p. 46). So we assume that $\psi$ is in the null-space of $G^*$ which implies that 
$$v = \dnu{v} = 0 \,  \text{ on } \, \Gamma_C \quad \text{ giving that } \quad v = 0 \,  \text{ in } \, \domega$$
where we again appeal to Holmgren's Theorem proving the claim by the Trace Theorem. 
\end{proof}

Now that we have analyzed the operator $G$ we will turn our attention to studying $(L-L_0)$. This is the other operator used in our initial factorization of difference of the NtD operators. Notice that the dependance on the unknown region $\Omega$ is more explicit for these operators since they map to the traces of function on $\Gamma_N \subsetneq \partial \Omega$.

\begin{theorem}\label{opL1}
The operator $(L-L_0):L^2(\Gamma_N) \rightarrow H^{1/2}(\Gamma_C)$ as defined in \eqref{opl} is injective provided that $\gamma_{\mathrm{max}}$ is sufficiently small or $\gamma_{\mathrm{min}}$ is sufficiently large.
\end{theorem}
\begin{proof}
We begin by assuming $g$ is in the null-space of $(L-L_0)$ which implies that 
$$\Delta(u-u_0)= 0  \text{ in } \domega, \quad \dnu{(u-u_0)} = 0 \,  \text{ on } \,  \Gamma_N, \,  \text{ and } \,(u- u_0) = 0 \,  \text{ on } \,  \Gamma_C.$$
It is clear that the above boundary value problem only admits the trivial solution. Therefore, we have that $u=u_0$ in $\domega$ and hence $\partial_{\nu}{u_0}+\gamma u_0=0$ on $\Gamma_C$. Notice, that by \eqref{dpu0} we have that $u_0\in H^1(\Omega)$ is the solution of the boundary value problem
$$\Delta u_0 = 0 \text{ in } \Omega  \quad \text{with} \quad u_0 = 0 \quad  \text{on} \quad \Gamma_D, \quad \dnu{u_0} + \gamma u_0 = 0 \text{ on } \Gamma_C.$$ Recall, that $\nu$ is the inward unit normal to $\Omega$ (see Figure \ref{domain}). From Green's second identity applied to $u_0$ in $\Omega$ and the Trace Theorem, we have that
\begin{align*}
0 &= \int_\Omega |\nabla u_0|^2 \dd x - \int_{\Gamma_C} \gamma |u_0|^2 \dd s \\
& \geq \|\nabla u_0\|^2_{L^2(\Omega)} - \gamma_{\text{max}} \|u_0\|^2_{L^2(\Gamma_C)} \\
& \geq \|\nabla u_0\|^2_{L^2(\Omega)} - \gamma_{\text{max}} C \|u_0\|^2_{H^1(\Omega)},                                                                                                                                                                                                                                                                                                                                                                                                                                                                        
\end{align*}
since $\gamma \leq \gamma_{\text{max}}$ a.e. on $\Gamma_C$ from our assumptions. Notice, that since $u_0|_{\Gamma_D}=0$ we have the  Poincar\'{e} estimate $\|u_0\|_{H^1(\Omega)} \leq C\| \nabla u_0\|_{{L^2(\Omega)}}$. Therefore,
\begin{align*}
0 \geq (1-\gamma_{\text{max}} C) \|\nabla u_0\|^2_{{L^2(\Omega)}}
\end{align*}
which implies that if $\gamma_{\text{max}}$ is small enough, then $|\nabla u_0| = 0$ in $\Omega$ and hence $u_0 = 0$ in $\Omega$ due to the zero trace on $\Gamma_D$. By the unique continuation principle (see for e.g. \cite{Colto2013}, p. 276), we obtain that $u_0 = 0$ in $D$, which implies that $g=0$ by the Trace Theorem. The other case can be proven similarly by considering the opposite sign of the above equality.   
\end{proof}

With this, we wish to prove that $(L-L_0)$ is compact with a dense range just as we did for the operator $G$. Note that the compactness is not obvious as in the previous case and to prove the density of the range we need to compute the adjoint operator $(L-L_0)^*$. To this end, let us consider the solution $p\in H^1(\domega)$ to 
\begin{align}\label{auxp}
\Delta p = 0 \,  \text{ in } \,\domega \quad  \text{with} \quad  \dnu{p}  = 0 \,  \text{ on } \, \Gamma_N , \quad  \dnu{p} + \gamma p = \xi \,  \text{ on } \, \Gamma_C
\end{align}
and the solution $q \in H^1(D)$ to 
\begin{align}\label{auxq}
\Delta q = 0 \text{ in } D\setminus \Gamma_C\quad  \text{with}\quad  \dnu{q}|_{\Gamma_N} = 0, \quad q|_{\Gamma_D}=0,\quad  \text{and} \quad [\![\dnu q]\!] |_{\Gamma_C} =  \xi
\end{align}
for any $\xi \in \widetilde{H}^{-1/2}(\Gamma_C)$. Here, we define the notation 
$$ [\![\dnu q]\!]  |_{\Gamma_C} = (  \partial_{\nu}{q^+} -  \partial_{\nu}{q^-}) |_{\Gamma_C},$$ 
where $+$ and $-$ indicate the limit obtained by approaching the boundary $\Gamma_C$ from $\domega$ and $\Omega$, respectively.  Note that since $q\in H^1(D)$ it has continuous trace across $\Gamma_C$. With this we can further analyze the operator $(L-L_0 )$.

\begin{theorem} \label{opL2}
The adjoint $(L-L_0)^*: \widetilde{H}^{-1/2}(\Gamma_C) \to  L^2(\Gamma_N) $ is given by 
$$(L-L_0)^* \xi  = (p-q)|_{\Gamma_N},$$
where $p $ and $q$ are the solution to \eqref{auxp} and \eqref{auxq}, respectively. Moreover, the operator $(L-L_0)$ is compact with a dense range provided that $\gamma_{\mathrm{max}}$ is sufficiently small or $\gamma_{\mathrm{min}}$ is sufficiently large.
\end{theorem}
\begin{proof}
To prove the claim, we first compute the adjoints $L^*$ and $L^*_0$ separately. We begin with computing $L^*$. Just as in Theorem \ref{opG2} we use Green's second identity to obtain that 
$$ 0 =  \int_{\Gamma_N} {p} \, \partial_{\nu}u - u \, \partial_{\nu} {p} \, \text{d}s + \int_{\Gamma_C} {p} \, \partial_{\nu}u - u \, \partial_{\nu} {p} \, \text{d}s,$$
where we have used the fact that the functions are both harmonic. By the boundary conditions in \eqref{dp1} and \eqref{auxp} we have that 
$$ \big( g , L^*\xi \big)_{L^2(\Gamma_N)} =\int_{\Gamma_N} {p} {g} \, \text{d}s =  \int_{\Gamma_C} {u} \, {[ \partial_{\nu} {p}  +\gamma {p}]} \, \text{d}s = \int_{\Gamma_C} {u}{\xi} \dd s =\big\langle L g , \xi \big\rangle_{\Gamma_C},$$ 
which gives $L^*\xi  = p|_{\Gamma_N}$. 

Now, for computing $L^*_0$ we proceed is a similar manner where we apply Green's second identity in $\Omega$ and $\domega$ to obtain that
\begin{align*}
0 = - \int_{\Gamma_C} {u_0} \partial_{\nu}{q^-} - \partial_{\nu}{{u_0}}q \dd s \quad \text{and}\quad 0= - \int_{\Gamma_N} \partial_{\nu}{{u_0}} q \dd s + \int_{\Gamma_C} {u_0}  \partial_{\nu}{q^+} -  \partial_{\nu}{{u_0}} q \dd s,
\end{align*}
where we have used that the functions are harmonic as well as $\partial_{\nu} q =0$ on $\Gamma_N$ and $q = u_0=0$ on $\Gamma_D$. By adding the above equations and using the boundary conditions in \eqref{dpu0} and \eqref{auxq} we obtain that
\begin{align*}
\big\langle \xi, L_0 g \big\rangle_{\Gamma_N}= \int_{\Gamma_C} {u_0} \xi \dd s = \int_{\Gamma_N} {g} q \dd s = (L_0^*\xi , g)_{L^2(\Gamma_N)},
\end{align*}
which gives $L_0^*\xi  = q|_{\Gamma_N}$.

With this, it is clear that $(L-L_0)^*$ is compact by the compact embedding of $H^{1/2}(\Gamma_N)$ into $L^2(\Gamma_N)$ which implies that $(L-L_0)$ is compact. Now, let $\xi$ be in the null-space of $(L-L_0)$ which gives that 
$$\Delta(p-q)= 0  \text{ in } \domega, \quad (p-q)=\dnu{(p-p_0)} = 0 \,  \text{ on } \,  \Gamma_N.$$
Therefore, by Holmgren's Theorem we have that $p=q$ in $\domega$. By the boundary conditions on $\Gamma_C$ 
$$ \partial_{\nu} {p}  +\gamma {p} = \xi = \partial_{\nu}{q^+} -  \partial_{\nu}{q^-} \quad \text{on} \quad \Gamma_C$$
which implies that 
$$\Delta q = 0 \text{ in } \Omega  \quad \text{with} \quad q = 0 \text{ on } \Gamma_D, \quad \dnu{q^{-}} + \gamma q = 0 \text{ on } \Gamma_C.$$
Here, we have used that  $\partial_{\nu} {p} =\partial_{\nu}{q^+}$ and $p=q$ on $\Gamma_C$.  Then, by arguing just as in Theorem \ref{opL1} we have that $q=0$ provided that $\gamma_{\text{max}}$ is small enough or $\gamma_{\mathrm{min}}$ is sufficiently large, which gives that $\xi=0$. 
\end{proof}

\subsection{The Linear Sampling Method}\label{LSMsect}
Now that we have the above results we can infer the analytical properties of the difference of the NtD operators $(\Lambda-\Lambda_0 )$. These properties of the operator are essential for applying the linear sampling method (LSM) for solving the inverse shape problem. This method has been used to solve many inverse shape problems (see for e.g. \cite{GLSM,eit-harris}). This method connects the unknown region to range of the data operator $(\Lambda-\Lambda_0 )$ via the solution to an ill-posed operator equation. To proceed, we will discuss the necessary analysis to show that the linear sampling method can be applied to this problem. From the analysis in the previous section, we have the following result for the difference of the NtD operators. 

\begin{theorem}\label{opNtD1}
The difference of the NtD operators $(\Lambda-\Lambda_0 ): L^2(\Gamma_N)  \to  L^2(\Gamma_N) $ given by \eqref{NtDop} has the factorization 
$$ (\Lambda-\Lambda_0) = G(L-L_0),$$ 
where $G$ and $(L-L_0)$ are defined in \eqref{opg} and \eqref{opl}, respectively. 
Moreover, the operator $(\Lambda-\Lambda_0 )$ is compact and injective with a dense range provided that $\gamma_{\mathrm{max}}$ is sufficiently small or $\gamma_{\mathrm{min}}$ is sufficiently large. 
\end{theorem}
 
To proceed, we need to determine an associated function that depends on the sampling point $z \in D$ to derive a `range test' to reconstruct the unknown subregion $\Omega$. To this end, we define the mixed Green's function (also referred to as the Zaremba function \cite{Zaremba}, p. 
B209): for any $z \in D$, let $\mathbb{G}(\cdot, z) \in H_{loc}^1(D\setminus \{z\})$ be the solution to
\begin{align}\label{solG}
-\Delta \mathbb{G}(\cdot, z) = \delta(\cdot -z) \text{ in } D,  \quad \dnu{\mathbb{G}(\cdot, z)} = 0 \text{ on } \Gamma_N, \quad
  \text{ and } \quad \mathbb{G}(\cdot, z) = 0 \text{ on } \, {\Gamma_D}.
\end{align}
The following result shows that the range of the operator $G$ given by \eqref{opg} uniquely determines the region of interest $\Omega$.
\begin{theorem}\label{ranG}
Let $G: H^{1/2}(\Gamma_C) \to L^2(\Gamma_N)$ as defined in \eqref{opg}. Then,
\begin{align*}
\mathbb{G}(\cdot, z)|_{\Gamma_N} \in \mathrm{Range}(G) \quad  \text{if and only if} \quad z \in \Omega.
\end{align*}
\end{theorem}
\begin{proof}
To prove the claim, we first start with the case when the sampling point $z \in {\Omega}$. With this we see that $\mathbb{G}(\cdot, z) \in H^1(\domega)$ satisfies
$$\Delta \mathbb{G}(\cdot, z) =0 \text{ in } \domega,  \quad \dnu{\mathbb{G}(\cdot, z)} = 0 \text{ on } \Gamma_N, \quad
\text{ and } \quad \mathbb{G}(\cdot, z)|_{\Gamma_C} := \varphi_z \in H^{1/2}(\Gamma_C).$$
From this we obtain that $G\varphi_z = \mathbb{G}(\cdot, z)|_{\Gamma_N}$ proving this case.   

Now, we consider the case when $z \in \domega$ and we proceed by contradiction. To this end, we assume that there is a $w_z \in H^1(\domega)$ such that 
$$\Delta w_z =0 \text{ in } \domega,  \quad \dnu{w_z} = 0 \text{ on } \Gamma_N, \quad
\text{ and } \quad w_z = \varphi_z \text{ on } \Gamma_C$$
for some $\varphi_z \in H^{1/2}(\Gamma_C)$ where $w_z = \mathbb{G}(\cdot, z)$ on $\Gamma_N$. By appealing to Holmgren's Theorem we can obtain that $w_z = \mathbb{G}(\cdot, z)$ in the set $( \domega ) \setminus\{z\}$. Using interior elliptic regularity 
(see \cite{salsa} p. 536) we have that  $w_z$ is continuous at the sampling point $z$. By the singularity at $z$ for the mixed Green's function $\mathbb{G} ( \cdot ,z)$ we have that  
$$| w_z (x)| < \infty \quad \text{ whereas}  \quad | \mathbb{G} (x,z) | \rightarrow \infty \quad \text{as} \quad x \rightarrow z.$$
This proves the claim by contradiction. 
\end{proof}

With the result in Theorem \ref{ranG} we can prove that the linear sampling method can be used to recover $\Omega$ from the NtD mapping. This is useful in non-destructive testing since there is no initial guess needed for this algorithm. With this, we can now state the main result in this subsection. 
\begin{theorem}\label{lsm}
Let the difference of the NtD operators $(\Lambda-\Lambda_0 ): L^2(\Gamma_N)  \to  L^2(\Gamma_N) $ be given by \eqref{NtDop}. Then for any sequence $\big\{g_{z , \ep} \big\}_{\ep >0} \in L^2(\Gamma_N)$ for $z \in D$ satisfying 
$$\big\| (\Lambda - \Lambda_0) g_{z , \ep} - \mathbb{G}(\cdot, z)|_{\Gamma_N}  \big\|_{ L^2(\Gamma_N)} \longrightarrow 0 \quad \text{as } \ep \rightarrow 0$$
we have that $\| g_{z , \ep} \|_{ L^2(\Gamma_N)} \longrightarrow \infty$ as $\ep \rightarrow 0$ for all $z \notin \Omega$ provided that $\gamma_{\mathrm{max}}$ is sufficiently small or $\gamma_{\mathrm{min}}$ is sufficiently large.
\end{theorem}
\begin{proof}
To prove the claim, we first note that by Theorem \ref{opNtD1} we have that $(\Lambda-\Lambda_0 )$ has a dense range in $L^2(\Gamma_N)$. Therefore, for all $z \in D$ we have that there exists an approximating sequence $\big\{g_{z , \ep} \big\}_{\ep >0}$ such that $(\Lambda - \Lambda_0) g_{z , \ep}$ converges in norm to $\mathbb{G}(\cdot, z)|_{\Gamma_N}$. For a contradiction, assume that there is such an approximating sequence such that $\| g_{z , \ep} \|_{L^2(\Gamma_N)}$ is bounded as $\ep \rightarrow 0$. Then we can assume that (up to a subsequence) it is weakly convergent such that $g_{z , \ep} \weakc g_{z,0}$ as $\ep \to 0$. By the compactness of the operator $(\Lambda - \Lambda_0)$ we have that as $\ep \to 0$
$$ (\Lambda - \Lambda_0) g_{z , \ep}  \longrightarrow (\Lambda - \Lambda_0) g_{z,0} \quad \text{which implies that} \quad (\Lambda - \Lambda_0) g_{z , 0} = \mathbb{G}(\cdot, z)|_{\Gamma_N}.$$
By the factorization $(\Lambda-\Lambda_0) = G(L-L_0)$ this would imply that $\mathbb{G}(\cdot, z)|_{\Gamma_N} \in \mathrm{Range}(G)$. This clearly contradicts Theorem \ref{ranG} proving the claim by contradiction.
\end{proof}

Notice, we have shown that the linear sampling method can be used to recover $\Omega$ in Theorem \ref{lsm}. In order to use this result to recover the corroded part of the region we find an approximate solution to 
\begin{align}\label{volt-gap}
(\Lambda - \Lambda_0) g_z = \mathbb{G}(\cdot, z)|_{\Gamma_N}.
\end{align}
Since, the operator $(\Lambda - \Lambda_0)$ is compact this implies that the above equation is ill-posed. But the fact that $(\Lambda - \Lambda_0)$ has a dense range means we can construct an approximate solution using a regularization strategy (see for e.g. \cite{kirschipbook}). Here, we can take $\ep>0$ to be the regularization parameter then we can recover $\Omega$ by plotting the imaging functional 
$$W_{\mathrm{LSM}}(z) = 1/\| g_{z , \ep} \|_{L^2(\Gamma_N )}$$
 where $ g_{z , \ep}$ is the regularized solution to \eqref{volt-gap}. Theorem \ref{lsm} implies that $W_{\mathrm{LSM}}(z) \approx 0$  for any $z \notin \Omega$. Note that we can not infer that $W_{\mathrm{LSM}}(z)$ will be bounded below in $z \in \Omega$. Therefore, we will consider the factorization method in the proceeding section.

\subsection{The Factorization Method}\label{FMsect}
In this section, we will consider using the factorization method (FM) to recover the corroded region $\Omega$. Even though we have already studied the linear sampling method, we see that Theorem \ref{lsm} does not prove that the corresponding imaging functional is bounded below for $z \in \Omega$. With this in mind, we consider the factorization method since it gives an exact characterization of the region of interest $\Omega$ using the spectral decomposition of an operator associated to $(\Lambda - \Lambda_0)$. 

To begin, we need to derive a `symmetric' factorization of the operator $(\Lambda - \Lambda_0)$. Therefore, we recall that by Theorem \ref{opNtD1} we have that $(\Lambda - \Lambda_0)=G(L - L_0)$. Now, we define the bounded linear operator 
\begin{align}\label{opt}
T: \widetilde{H}^{-1/2}(\Gamma_C) \rightarrow H^{1/2}(\Gamma_C) \quad \text{given by} \quad T\xi = (p-q)|_{\Gamma_C}.
\end{align}
where $p$ and $q$ satisfy \eqref{auxp} and \eqref{auxq}, respectively. It is clear that the boundedness of $T$ follows from the well-posedness of \eqref{auxp} and \eqref{auxq} along with the Trace Theorem. With this, we notice that 
$$GT\xi = G\left((p-q)|_{\Gamma_C}\right) = (p-q)|_{\Gamma_N} = (L-L_0)^*\xi $$
where we have used the fact that $(p-q) \in H^1(\domega)$ satisfies
$$\Delta (p-q)= 0  \text{ in } \domega, \quad \dnu{(p-q)} = 0 \,  \text{ on } \, \Gamma_N, \,  \text{ and } \,(p-q)|_{\Gamma_C} \in H^{1/2}(\Gamma_C)$$
along with the definition of $G$ in \eqref{opg} and $(L-L_0)^*$ given in Theorem \ref{opL2}. Since this is true for any $\xi \in \widetilde{H}^{-1/2}(\Gamma_C)$ we have that $GT = (L-L_0)^*$. We now have the desired factorization of $(\Lambda - \Lambda_0)=G T^* G^*$ by the calculation that $(L-L_0) = T^* G^*$.

With this new factorization acquired we now prove that $(\Lambda - \Lambda_0)$  is self-adjoint. This would imply that $(\Lambda - \Lambda_0)=G T G^*$ where $G$ and $T$ are defined in \eqref{opg} and \eqref{opt}, respectively. To this end, we notice that  
$$\big(g_1, (\Lambda - \Lambda_0)g_2\big)_{L^2(\Gamma_N)} = \int_{\Gamma_N} g_1 {\left[{u^{(2)}}-{u^{(2)}_0}\right]} \dd s =  \int_{\Gamma_N} {u^{(2)}} \dnu{u^{(1)}}  - {u^{(2)}_0} \dnu{u^{(1)}_0}\dd s,$$
where the superscript ${(j)}$ corresponds to the solution of \eqref{dp1} and \eqref{dpu0} for $g_j \in L^2(\Gamma_N)$ with $j=1,2$. We now apply Green's first identity to $u^{(j)}$ and $u^{(j)}_0$ in $\domega$ and $D$, respectively to obtain
$$ \big(g_1, (\Lambda - \Lambda_0)g_2\big)_{L^2(\Gamma_N)}=   \int_{\domega} \nabla u^{(1)} \cdot \nabla { u^{(2)} } \dd x + \int_{\Gamma_C} \gamma u^{(1)} { u^{(2)} }  \dd s - \int_D \grad u^{(1)}_0 \cdot \grad { u^{(2)}_0 }\dd x, $$
where we have used the boundary conditions. This implies that $(\Lambda-\Lambda_0)=G T G^*$ is self-adjoint.

In order to apply the theory of the factorization method \cite{regfm1,kirschbook} we need to study the operator $T$ defined in \eqref{opt}. In particular, we wish to show that under some assumptions that $\pm T$ is coercive on the range of $G^*$. This can be achieved by showing that $\pm T$ is a coercive operator from $\widetilde{H}^{-1/2}(\Gamma_C)$ to $H^{1/2}(\Gamma_C)$. With this in mind, notice that by the boundary conditions on the corroded boundary $\Gamma_C$ we have that 
$$\big\langle T \xi , \xi \big\rangle_{\Gamma_C} = \int_{\Gamma_C} {(p-q)}{\xi} \dd s = \int_{\Gamma_C} p {[ \partial_{\nu} {p}  +\gamma {p}]} - q [\![\dnu q]\!] \dd s$$
and by appealing to Green's first identity 
\begin{align}\label{vart}
\big\langle T \xi , \xi \big\rangle_{\Gamma_C} = \int_{\domega} |\grad p|^2 \dd x + \int_{\Gamma_C}\gamma |p|^2 \dd s - \int_{D} |\grad q|^2 \dd x.
\end{align}
By \eqref{vart} we see that there is no way for $\pm T$ to be coercive without some extra assumptions because of the negative multiplying the $L^2(D)$--norm of the gradient of $q$. Therefore, to proceed we will consider two cases $0< \gamma <1$  or $1<\gamma$ a.e. on $\Gamma_C$.

For the first case when $0< \gamma <1$ a.e. on $\Gamma_C$ we have that 
\begin{align*}
\big\langle T \xi , \xi \big\rangle_{\Gamma_C} &\geq  \gamma_{\mathrm{min}} \left[ \|\grad p\|^2_{L^2(\domega)} + \|p\|_{L^2(\Gamma_C)}^2 \right] - \|\grad q\|_{L^2(D)}^2 \quad \text{since $0< \gamma <1$}   \\
								     &\geq \gamma_{\mathrm{min}} \left[ \|\grad p\|_{L^2(\domega)}^2 + \|p\|_{L^2(\Gamma_C)}^2 \right] - C\| \xi \|_{H^{-1/2}(\Gamma_C)}^2 \quad \text{by the well-posedness.}
\end{align*}
Now, notice that by the boundary condition in \eqref{auxp} we can estimate 
\begin{align*}
\| \xi \|_{H^{-1/2}(\Gamma_C)}&= \|  \partial_{\nu} {p}  +\gamma {p}\|_{H^{-1/2}(\Gamma_C)} \\
					     &\leq \|  \partial_{\nu} {p} \|_{H^{-1/2}(\Gamma_C)} + \|{p}\|_{H^{1/2}(\Gamma_C)} \quad \text{since $0< \gamma <1$} \\
					     &\leq C \|{p}\|_{H^{1}(\domega)} \quad \text{by Trace Theorems}\\
					     &\leq C \sqrt{ \|\grad p\|^2_{L^2(\domega)} + \|p\|_{L^2(\Gamma_C)}^2}
\end{align*}
where we have used that 
$$ \|{p}\|^2_{H^{1}(\domega)} \quad \text{is equivalent to } \quad  \|\grad p\|^2_{L^2(\domega)} + \|p\|_{L^2(\Gamma_C)}^2.$$
With this we see that $\exists \, C_j>0$ independent of $\gamma$ for $j=1,2$ where 
$$\big\langle T \xi , \xi \big\rangle_{\Gamma_C} \geq (C_1\gamma_{\mathrm{min}} -C_2 )\| \xi \|_{H^{-1/2}(\Gamma_C)}^2.$$
This implies that for $C_2/C_1 < \gamma <1$ a.e. on $\Gamma_C$ then we have that $T$ is coercive. 

Now, for the case $1<\gamma$ a.e. on $\Gamma_C$ we have that 
$$\int_{\domega} |\grad p|^2 \dd x + \int_{\Gamma_C}\gamma |p|^2 \dd s =\int_{\Gamma_C} \xi {p} \dd s$$
by Green's first identity. With this, by the Trace Theorem we have the estimate 
$$\|\grad p\|^2_{L^2(\domega)} + \|p\|_{L^2(\Gamma_C)}^2 \leq C \| \xi \|_{H^{-1/2}(\Gamma_C)}\|{p}\|_{H^{1}(\domega)} .$$
By the aforementioned norm equivalence, we further establish that
$$ \sqrt{ \|\grad p\|^2_{L^2(\domega)} + \|p\|_{L^2(\Gamma_C)}^2} \leq C \| \xi \|_{H^{-1/2}(\Gamma_C)}.$$
Using the boundary condition in \eqref{auxq} we have that
\begin{align*}
\| \xi \|_{H^{-1/2}(\Gamma_C)}&= \|  \partial_{\nu}{q^+} -  \partial_{\nu}{q^-}\|_{H^{-1/2}(\Gamma_C)} \\
					     &\leq C \left[ \|  q \|_{H^{1}(\domega)} + \|  q \|_{H^{1}(\Omega)}  \right] \quad \text{by Trace Theorem} \\
					     &\leq C \|\grad {q}\|_{L^2(D)} \quad \text{by the Poincar\'{e} estimate since $q|_{\Gamma_D} = 0$.}
\end{align*}
Therefore, by \eqref{vart} we have that $\exists \, C_j>0$ independent of $\gamma$ for $j=3,4$ where 
\begin{align*}
- \big\langle T \xi , \xi \big\rangle_{\Gamma_C}&\geq \|\grad {q}\|^2_{L^2(D)}  - \gamma_{\mathrm{max}} \left[ \|\grad p\|^2_{L^2(\domega)} + \|p\|_{L^2(\Gamma_C)}^2 \right]  \\
                                                                         &\geq (C_3 -C_4  \gamma_{\mathrm{max}})\| \xi \|_{H^{-1/2}(\Gamma_C)}^2.
\end{align*}
This implies that for $C_3/C_4 > \gamma >1$ a.e. on $\Gamma_C$ then we have that $-T$ is coercive. 

Even though we have proven the coercivity it is unclear if the assumption that 
$$C_2/C_1 < \gamma <1 \quad \text{ or } \quad C_3/C_4 > \gamma >1 \quad \text{a.e. on $\Gamma_C$} $$ 
is satisfied. This is due to the fact that the constants $C_j$ are unknown and depend on the geometry. In order to continue in our investigation, we make the assumption that there exists regions $\domega$ and $\Omega$ such that the above assumptions are valid for some given $\gamma$. With this we have the main result of this subsection. 

\begin{theorem}\label{fm}
Let the difference of the NtD operators $(\Lambda-\Lambda_0 ): L^2(\Gamma_N)  \to  L^2(\Gamma_N) $ be given by \eqref{NtDop}.  Provided that either $\pm T$ defined by \eqref{opt} is coercive, then 
$$\mathbb{G}(\cdot, z)|_{\Gamma_N} \in \mathrm{Range}\big( |\Lambda-\Lambda_0 |^{1/2} \big) \quad  \text{if and only if} \quad z \in \Omega.$$
\end{theorem}
\begin{proof}
This is due to the fact that with the factorization $(\Lambda - \Lambda_0)=G T G^*$ and provided that $\pm T$ is coercive we have that $\mathrm{Range}\big( |\Lambda-\Lambda_0 |^{1/2} \big) = \mathrm{Range}(G)$ by the result in \cite{embry,regfm1}. Here, we note that $|\Lambda-\Lambda_0 |^{1/2}$ is defined in the standard way by the spectral decomposition of a self-adjoint compact operator. Then by appealing to Theorem \ref{ranG} proves the claim. 
\end{proof}

With this result, we have another way to recover the corroded region $\Omega$. Notice that since $(\Lambda - \Lambda_0)$ is a self-adjoint compact operator Theorem \ref{fm} can be reformulated as 
$$ \sum\limits_{j=1}^{\infty} \frac{1}{\sigma_j} \big|\big( \mathbb{G}(\cdot, z) \, , \, g_j \big)_{L^2(\Gamma_N)}\big|^2 < \infty \quad  \text{if and only if} \quad z \in \Omega $$
by appealing to Picard's criteria (see for e.g. \cite{kirschipbook,kirschbook}) where $ (\sigma_j , g_j ) \in \R_{+} \times L^2(\Gamma_N)$ is the eigenvalue decomposition of the absolute value for the difference of the NtD operators. This result is stronger than Theorem \ref{lsm} since the result is an equivalence that implies that $\Lambda$ uniquely determines the subregion $\Omega$. Also, to numerically recover $\Omega$ we can use the imaging functional 
$$W_{FM}(z) = \left[ \sum\limits_{j=1}^{\infty} \frac{1}{\sigma_j} \big|\big( \mathbb{G}(\cdot, z) \, , \, g_j \big)_{L^2(\Gamma_N)}\big|^2 \right]^{-1}$$
which is positive only when $z \in \Omega$. Since $(\Lambda - \Lambda_0)$ is compact we have that the eigenvalues $\sigma_j$ tend to zero rapidly which can cause instability in using the imaging functional $W_{FM}(z)$. In \cite{regfm1,regfm2} it has been shown that adding a regularizer to the sum can regain stability while still given the unique reconstruction of $\Omega$.

\section{Inverse Impedance Problem}\label{sect-ip2}
In this section, we consider the inverse impedance problem, i.e. determine the corrosion parameter $\gamma$ on $\Gamma_C$ from the knowledge NtD mapping $\Lambda_\gamma$. Here, we will assume that the corroded boundary $\Gamma_C$ is known. This would be the case, if it was reconstructed as discussed in the previous section. We will prove that $\gamma \mapsto \Lambda_\gamma$ is injective as a mapping from $L^{\infty}(\Gamma_C)$ into $\mathscr{L}(L^2(\Gamma_N))$ i.e the set of bounded linear operator acting on $L^2(\Gamma_N)$. Then we will prove a Lipschitz--stability estimate for the inverse impedance problem. Similar result have been proven in \cite{EIT-finiteElectrode,eit-transmission1,JIIP,invrobin-Meftahi} just to name a few recent works. This will imply that one can reconstruct $\gamma$ on $\Gamma_C$ from the known Cauchy data $g$ and $\Lambda_\gamma g$ on $\Gamma_N$. In order to show the uniqueness, let us first consider the following density result associated with solutions to \eqref{dp1}.

\begin{lemma}\label{density}
Let 
$$\mathcal{U} = \left\{ u|_{\Gamma_C} \in L^2(\Gamma_C) : u \in H^1(\domega) \textrm{ solves } \eqref{dp1} \textrm{ for any  } g \in L^2(\Gamma_N)\right\}.$$
Then, $\mathcal{U}$ is dense subspace in $L^2(\Gamma_C)$.
\end{lemma}
\begin{proof}
It is enough to show that $\mathcal{U}^\perp$ is trivial. To this end, notice that for any $\phi \in \mathcal{U}^\perp$ there exists $v \in H^1(\domega)$ that is the unique solution of
$$\Delta v= 0 \quad  \text{in } D\setminus \overline{\Omega}, \quad  \dnu{v} = 0 \quad  \text{on } \Gamma_N, \, \text{ and } \, \dnu{v} +\gamma v  = \phi \quad  \text{on }\Gamma_C .$$
From the boundary conditions, we have that 
\begin{align*}
0 = \int_{\Gamma_C}u {\phi} \dd s = \int_{\Gamma_C} u{(\dnu{v} + \gamma  v)} \dd s = \int_{\Gamma_C} u\dnu{{v}} - {v} \dnu{u}  \dd s.
\end{align*}
Then, by appealing to Green's second identity in $\domega$ we obtain 
\begin{align*}
0&= - \int_{\Gamma_N} u\dnu{{v}} - {v} \dnu{u}   \dd s = \int_{\Gamma_N} g {v} \dd s, \quad \text{for any } g \in L^2(\Gamma_C)
\end{align*}
where we have used that both $u$ and $v$ are harmonic in $\domega$. Therefore, $v|_{\Gamma_N}=0$ and $\dnu{v}|_{\Gamma_N}=0$ from the boundary condition, so we conclude that $v$ vanishes in $\domega$ by Holmgren's theorem. Hence, $\phi=0$ on $\Gamma_C$ by the Trace Theorem.
\end{proof}

Now, we will show that the NtD operator $\Lambda$ uniquely determines the boundary coefficient $\gamma$ on  $\Gamma_C$. To this end, consider the solutions $u$ and $u_0$ to \eqref{dp1} and \eqref{dpu0}, respectively and let $\mathbb{G}(\cdot, z)$ be the mixed Green's function defined in \eqref{solG}. Then, the following lemma allows one to rewrite $(u-u_0)(z)$ for any $z \in \domega$ in terms of a boundary integral operator. 

\begin{lemma}\label{dens}
For any $z\in \domega$,
\begin{align}\label{u-u0}
 -(u-u_0)(z) = \int_{\Gamma_C} u(x) \big[\dnu{\mathbb{G}(x, z)}  +  \gamma(x) {\mathbb{G}(x, z)} \big] \dd s(x).
\end{align}
\end{lemma}
\begin{proof}
For any $z \in \domega$, from the boundary conditions and Green's second identity,
\begin{align*}
 -(u-u_0)(z)  &=\int_{\domega} (u-u_0) \Delta \mathbb{G}(\cdot , z)  - \mathbb{G}(\cdot , z)  \Delta(u-u_0)  \dd x \\
  &= \int_{\Gamma_C}  (u-u_0)\dnu{\mathbb{G}(\cdot , z)}- {\mathbb{G}(\cdot , z)} \dnu{(u-u_0)} \dd s \\
  &= \int_{\Gamma_C} u \dnu{\mathbb{G}(\cdot , z)}- {\mathbb{G}(\cdot , z)}  \dnu{u} \dd s - \int_{\Gamma_C} u_0 \dnu{\mathbb{G}(\cdot , z)}-  {\mathbb{G}(\cdot, z)} \dnu{u_0} \dd s.
\end{align*}
Applying Green's second identity to $u_0$ and $\mathbb{G}(\cdot , z)$ in $\Omega$,
\begin{align*}
- \int_{\Gamma_C} u_0\dnu{\mathbb{G}(\cdot , z)} &- {\mathbb{G}( \cdot , z)} \dnu{u_0}  \dd s= \int_{\Gamma_D} u_0\dnu{\mathbb{G}(\cdot , z)}- {\mathbb{G}( \cdot , z)} \dnu{u_0}\dd s = 0,
\end{align*}
where we have used the fact that  $u_0$ and $\mathbb{G}(\cdot , z)$ have zero trace on $\Gamma_D$ which completes the proof. 
\end{proof}

The result in Lemma \ref{dens} will now be used to prove that the NtD operator uniquely determines the corrosion coefficient $\gamma$. We would like to also note that the representation formula above can be used as an integral equation to solve for $\gamma$. Assuming that the Cauchy data for $u$ is known on $\Gamma_N$, we can recover the Cauchy data on $\Gamma_C$ numerically as in \cite{Data-completion}. Therefore, by restricting the representation formula in Lemma \ref{dens} onto $\Gamma_C$ (or $\Gamma_N$) gives an integral equation for the unknown coefficient. We now prove our uniqueness result.

\begin{theorem}
Assume that $\gamma \in L^{\infty}(\Gamma_C)$ and satisfies the inequality in Section \ref{sect-dp}. Then, the mapping $\gamma \mapsto \Lambda_\gamma$ from $L^{\infty}(\Gamma_C) \to \mathscr{L}(L^2(\Gamma_N))$ is injective.
\end{theorem}
\begin{proof}
To prove the claim, let ${\gamma_j}$ for $j=1,2$ be the corrosion coefficient in \eqref{dp1} such that $\Lambda_{\gamma_1} = \Lambda_{\gamma_2}$ for any $g \in L^2(\Gamma_N)$. Then the corresponding solutions 
$$u_{\gamma_1}=u_{\gamma_2} \quad \text{and} \quad \dnu{u_{\gamma_1}}=\dnu{u_{\gamma_2}} \quad \text{on } \,\, \Gamma_N, $$ 
which implies that $u_{\gamma_1}=u_{\gamma_2}$ in $\domega$ for any $g \in L^2(\Gamma_N)$ by Holmgren's Theorem. If we denote $u_{\gamma_1}=u_{\gamma_2}$ by $u$, then by subtracting \eqref{u-u0} we have that 
$$0 = \int_{\Gamma_C} (\gamma_1-\gamma_2)(x) u(x)  \mathbb{G}(x, z) \dd s(x) \quad \text{ for all $ z \in \domega$ and for any $g \in L^2(\Gamma_N)$. }$$
From Lemma \ref{density}, we obtain that $(\gamma_1-\gamma_2)(x) \mathbb{G}(x, z) = 0$ for a.e. $x \in \Gamma_C$ and for all $ z \in \domega$. Notice that by interior elliptic regularity for any $x \in D\setminus \{z\}$ the mixed Green's is continuous at $x$.

Now, by way of contradiction assume that $\| \gamma_1-\gamma_2 \|_{L^{\infty}(\Gamma_C)} \neq 0$. This would imply that there exists a subset $\Sigma \subset \Gamma_C$ with positive boundary measure such that $|\gamma_1-\gamma_2| >0$ on $\Sigma$. Therefore, for some $x^* \in \Sigma$ we hat that $\mathbb{G}(x^* , z) = 0 $ for all $ z \in \domega$. Then, we can take a sequence $z_n \in \domega$ such that $z_n \to x^*$ as $n \to \infty$. This gives a contradiction since 
$$\mathbb{G}(x^* , z_n ) = 0 \, \, \text{ for all $n \in \N$} \quad \text{ and } \quad | \mathbb{G} (x^* , z_n ) | \rightarrow \infty \quad \text{as} \quad n \rightarrow \infty.$$
This implies that $\| \gamma_1-\gamma_2 \|_{L^{\infty}(\Gamma_C)} = 0$, proving the claim. 
\end{proof}

Now that we have proven our uniqueness result we turn our attention to proving a stability estimate. We will prove a  Lipschitz--stability estimate using similar techniques in \cite{eit-transmission1}. This will employ a monotonicity estimate for the NtD operator $\Lambda_\gamma$ with respect to the corrosion parameter $\gamma$ as well as our density result in Lemma \ref{density}. With this in mind, we now present the monotonicity estimate. 

\begin{lemma}\label{monotonicity}
Let the NtD operators $\Lambda_{\gamma_j}: L^2(\Gamma_N)  \to  L^2(\Gamma_N)$ be given by \eqref{NtDop} with corrosion parameter $\gamma_j \in L^{\infty}(\Gamma_C)$ for $j=1,2$ and satisfies the inequality in Section \ref{sect-dp}. Then, 
$$\int_{\Gamma_C}(\gamma_1 - \gamma_2) |u_{\gamma_2}|^2 \dd s
\geq \int_{\Gamma_N} g {({\Lambda_{\gamma_2} }- {\Lambda_{\gamma_1} })g} \dd s.$$
\end{lemma}
\begin{proof} The proof is identical to what is done in \cite{eit-transmission1} so we omit the proof to avoid repetition. 
\end{proof}

With this we are ready to prove our Lipschitz--stability estimate. We will show that the inverse of the nonlinear mapping  $\gamma \mapsto \Lambda_\gamma$ is Lipschitz continuous from the set of bounded linear operators $\mathscr{L}(L^2(\Gamma_N))$ to a finite dimensional subspace of $L^{\infty}(\Gamma_C)$. To this end, we let $\mathcal{A}$ be a finite dimensional subspace of $L^{\infty}(\Gamma_C)$ and define the compact set 
$$\mathcal{A}_{[\gamma_{\text{min}},\gamma_{\text{max}}]} = \big\{ \gamma \in \mathcal{A} : \gamma_{\text{min}} \leq \gamma \leq \gamma_{\text{max}} \text{ on } \Gamma_C \big\}. $$
This would imply that inverse impedance problem has a unique solution that depends continuously on the NtD mapping. This fits nicely with the results from the previous section that assuming the factorization method is valid the inverse shape problem has a unique solution that depends continuously on the NtD mapping. 

\begin{theorem}\label{stable}
Let the NtD operators $\Lambda_{\gamma_j}: L^2(\Gamma_N)  \to  L^2(\Gamma_N)$ be given by \eqref{NtDop} with corrosion parameter $\gamma_j \in \mathcal{A}_{[\gamma_{\mathrm{min}},\gamma_{\mathrm{max}}]}$ for $j=1,2$ and satisfies the inequality in Section \ref{sect-dp}. Then, 
$$\|\gamma_1-\gamma_2\|_{L^\infty(\Gamma_C)} \leq C \| \Lambda_{\gamma_1} - \Lambda_{\gamma_2}\|_{\mathscr{L}(L^2(\Gamma_N))},$$
where $C>0$ is independent of $\gamma_j \in \mathcal{A}_{[\gamma_{\mathrm{min}},\gamma_{\mathrm{max}}]}$. 
\end{theorem}
\begin{proof}
To prove the claim, notice that from Lemma \ref{monotonicity}, we have that 
\begin{align*}
- \int_{\Gamma_N} g {({\Lambda_{\gamma_2} g}- {\Lambda_{\gamma_1} g})} \dd s \geq \int_{\Gamma_C}(\gamma_2 - \gamma_1) |u_{\gamma_2}|^2 \dd s
\end{align*}
and interchanging the roles of $\gamma_1$ and $\gamma_2$, we obtain
\begin{align*}
\int_{\Gamma_N} g {({\Lambda_{\gamma_2} g}- {\Lambda_{\gamma_1} g})} \dd s \geq \int_{\Gamma_C}(\gamma_1 - \gamma_2) |u_{\gamma_1}|^2 \dd s.
\end{align*}
Therefore, we have that 
\begin{align*}
&\|\Lambda_{\gamma_1} - \Lambda_{\gamma_2}\|_{\mathscr{L}(L^2(\Gamma_C))} \\
&\hspace{0.5in}= \sup_{\|g\|=1} \left| \int_{\Gamma_N} g {({\Lambda_{\gamma_2} }- {\Lambda_{\gamma_1} })g} \dd s \right|
= \sup_{\|g\|=1} \max \left\{\pm \int_{\Gamma_N} g {({\Lambda_{\gamma_2} }- {\Lambda_{\gamma_1} })g} \dd s \right\}\\
& \hspace{0.5in} \geq \sup_{\|g\|=1} \max \left\{\int_{\Gamma_C}(\gamma_1 - \gamma_2) |u_{\gamma_1}|^2 \dd s,  \int_{\Gamma_C}(\gamma_2 - \gamma_1) |u_{\gamma_2}|^2 \dd s \right\}.
\end{align*}
Notice that we have used the fact that $\Lambda_{\gamma_j}$ is self-adjoint. Here we let $\| \cdot \|$ denote the ${L^2(\Gamma_N)}$--norm. This implies that
\begin{align*}
&\frac{\|\Lambda_{\gamma_1} - \Lambda_{\gamma_2}\|_{\mathscr{L}(L^2(\Gamma_C))}}{\|\gamma_1- \gamma_2\|_{L^\infty(\Gamma_C)}} \\
&\hspace{0.5in}\geq \sup_{\|g\|=1} \max \left\{\int_{\Gamma_C}\frac{(\gamma_1 - \gamma_2)}{\|\gamma_1- \gamma_2\|_{L^\infty(\Gamma_C)}} |u_{\gamma_1}|^2 \dd s,  \int_{\Gamma_C}- \frac{(\gamma_1 - \gamma_2)}{\|\gamma_1- \gamma_2\|_{L^\infty(\Gamma_C)}} |u_{\gamma_2}|^2 \dd s \right\}.
\end{align*} 
Provided that $\gamma_1 \neq \gamma_2$, we now let 
$$ \zeta= \frac{(\gamma_1 - \gamma_2)}{\|\gamma_1- \gamma_2\|_{L^\infty(\Gamma_C)}}$$
and define $\Psi : L^2(\Gamma_N ) \rightarrow \R$  given by
$$ \Psi(g; \zeta , \gamma_1 , \gamma_2) =  \max \left\{\int_{\Gamma_C} \zeta |u_{\gamma_1}^{(g)}|^2 \dd s,  \int_{\Gamma_C}- \zeta |u_{\gamma_2}^{(g)}|^2 \dd s \right\}.$$ 
Then, to complete the proof, it suffices to show that 
$$ \inf_{\substack{\zeta \in \mathcal{C},\\ \kappa_1, \kappa_2 \in \mathcal{A}_{[a,b]}}} \sup_{\|g\|=1} \Psi(g; \zeta , \kappa_1 , \kappa_2)>0$$ 
where $\mathcal{C} = \big\{ \zeta \in \mathcal{A} : \|\zeta\|_{L^\infty(\Gamma_C)}=1\big\}$. Notice that since $\mathcal{A}_{[\gamma_{\text{min}},\gamma_{\text{max}}]}$ and $\mathcal{C}$ are finite dimensional, we have that they are compact sets. 

To this end, since we have that $\|\zeta\|_{L^{\infty}(\Gamma_C)}=1$, then there exists a subset $\Sigma \subset \Gamma_C$ with positive boundary measure such that for a.e. $x \in \Sigma$ either $\zeta(x) \geq 1/2$ or $-\zeta(x) \geq 1/2$. Without loss of generality assume that $\zeta(x) \geq 1/2$  a.e. for $x \in \Sigma$ and the other case can be handled in a similar way. From Lemma \ref{density}, there exists a sequence $\{g_n\}_{n=1}^\infty \in L^2(\Gamma_N)$ such that the corresponding solution $u_{\gamma_1}^{(g_n)}$ of \eqref{dp1} satisfies 
\begin{align*}
u_{\gamma_1}^{(g_n)} \rightarrow \frac{2 \chi_\Sigma}{\sqrt{\displaystyle \int_\Sigma \dd s}} \quad \text{as}\quad n \to \infty \quad \text{ in the } L^2(\Gamma_C)\text{--norm}.
\end{align*}
With the above convergence we have that
\begin{align*}
\lim_{n \rightarrow \infty} \int_\Sigma | u_{\gamma_1}^{(g_n)} |^2\dd s = 4 \quad \text{ and }\quad \lim_{n \rightarrow \infty} \int_{\Gamma_C \setminus \Sigma} |u_{\gamma_1}^{(g_n)}|^2 \dd s = 0.
\end{align*}
Then, there exists $\hat{g} \in L^2(\Gamma_N)$ such that
\begin{align*}
 \int_\Sigma |u_{\gamma_1}^{(\hat{g})}|^2 \dd s \geq 2 \quad \text{ and }\quad \int_{\Gamma_C \setminus \Sigma} |u_{\gamma_1}^{(\hat{g})}|^2 \dd s \leq 1/2.
\end{align*}
If $\zeta(x) \geq 1/2$ for $x \in \Sigma$, then since $\zeta(x) \geq -1$ for $x \in \Gamma_C \setminus \Sigma$ we have the estimate 
\begin{align*}
\Psi(\hat{g}; \zeta , \gamma_1 , \gamma_2) &= \int_{\Gamma_C} \zeta |u_{\gamma_1}^{(\hat{g})}|^2 \dd s \geq \frac{1}{2} \int_\Sigma |u_{\gamma_1}^{(\hat{g})}|^2\dd s - \int_{\Gamma_C \setminus \Sigma} |u_{\gamma_1}^{(\hat{g})}|^2 \dd s \geq \frac{1}{2} .
\end{align*}
By the linearity of \eqref{dp1} we have that 
$$\Psi \left(\hat{g}/\| \hat{g} \| ; \zeta , \gamma_1 , \gamma_2\right) =  \Psi(\hat{g}; \zeta , \gamma_1 , \gamma_2) /\| \hat{g} \|^2 \geq \frac{1}{2\| \hat{g} \|^2 }>0$$
which implies that $\sup_{\|g\|=1} \Psi(g; \zeta , \gamma_1 , \gamma_2)\geq1/2$. Now by the proof of Theorem \ref{monotonicity} we have that the mapping 
$$ (\zeta , \kappa_1 , \kappa_2) \mapsto \sup_{\|g\|=1} \Psi(g; \zeta , \kappa_1 , \kappa_2)$$ 
is semi-lower continuous on the compact set $\mathcal{C}\times \mathcal{A}_{[\gamma_{\text{min}},\gamma_{\text{max}}]} \times \mathcal{A}_{[\gamma_{\text{min}},\gamma_{\text{max}}]}$. This implies that it obtains its global minimum which is strictly positive by the above inequality, proving the claim.
\end{proof}

With this result we have completed our analytical study of the inverse shape and inverse parameter problem. To reiterate, we have prove that the inverse shape and inverse parameter problem have uniquely solvable solutions given the full NtD mapping of $\Gamma_N$.

\section{Numerical results}\label{sect-numerics}
In this section, we provide numerical examples for the reconstruction of $\Gamma_C$ using the NtD mapping. To this end, we first derive the corresponding integral equations to obtain the solution $u_0$ and $u$ on $\Gamma_N$ for \eqref{dpu0} and \eqref{dp1}, respectively. For more details on the definition of the integral operators and their jump relations we refer the reader to \cite[Chapter 7]{atkinson1997}. Next, we explain how to obtain $\mathbb{G}(\cdotp,z)$ on $\Gamma_N$ for a given set of points $z$ (see equation \eqref{solG}). Then, we illustrate how to discretize the NtD operator $\Lambda-\Lambda_0$ using the Galerkin approximation in order to apply the LSM (see equation \eqref{volt-gap}) or the FM (see Theorem \ref{fm}). Finally, we provide some reconstructions using both FM and LSM, respectively.

In order to provide numerical evidence of the effectiveness of the sampling methods, we need the following definitions.
We define 
	\[\Phi(x,y)=-\log(|x-y|)/2\pi\,,\quad  x\neq y\] 
to be the fundamental solution of the Laplace equation in $\R^2$. Assume that $A \subset \R^2$ is an arbitrary domain with boundary $\partial A$.
The single-layer potential for the Laplace equation over a given boundary $\partial A$ is denoted by
\begin{eqnarray*}
	\mathrm{SL}^{\partial A} \left[\phi\right](x)&=&\int_{\partial A}
	\Phi(x, y) \, \phi(y) \;\mathrm{d}s(y)\,,\quad x\in A\,,
\end{eqnarray*}
where $\phi$ is some density function. Now, we let $\partial A=  \overline{\Gamma_\alpha} \cup \overline{\Gamma_\beta} $ with $\Gamma_\alpha \cap \Gamma_\beta =\emptyset$ be the boundary of the domain $A$. 
The single- and double-layer boundary integral operators over the boundary $\Gamma_i$ evaluated at a point of $\Gamma_j$ are given as
\begin{eqnarray*}
	\mathrm{S}^{\Gamma_i\rightarrow \Gamma_j}\left[\phi\vert_{\Gamma_i}\right](x)&=&\int_{\Gamma_i} \Phi(x,y)\phi(y)\;\mathrm{d}s(y)\,,\quad x\in\Gamma_j\,,\\
	{\mathrm{T}}^{\Gamma_i\rightarrow \Gamma_j}\left[\phi\vert_{\Gamma_i}\right](x)&=&\int_{\Gamma_i} \partial_{\nu_j(x)}\Phi(x,y)\phi(y)\;\mathrm{d}s(y)\,,\quad x\in\Gamma_j\,,
\end{eqnarray*}
where $i,j\in \{\alpha,\beta \}$. Here, $\partial_{\nu_j(x)}$ denotes the normal derivative, where $\nu_j(x)$ is the exterior normal at $x\in\Gamma_j$.

\subsection{Integral equation for computing $u_0$ on $\Gamma_N$}
We first consider the uncorroded (healthy) object $D$, refer also to \eqref{dpu0}. Now, we are in position to explain how to obtain $u_0$ at any point of $\Gamma_N$.
\begin{proposition}
	Let $D$ be the domain representing the uncorroded object with boundary $\overline{\Gamma_N} \cup \overline{\Gamma_D}$ satisfying $\Gamma_N\cap \Gamma_D=\emptyset$. Then, the solution $u_0$ to \eqref{dpu0} on $\Gamma_N$ for the uncorroded object is given by
	\begin{eqnarray}
		u_0\vert_{\Gamma_N}(x)=\mathrm{S}^{\Gamma_N\rightarrow \Gamma_N}\left[\phi_0\vert_{\Gamma_N}\right](x)+\mathrm{S}^{\Gamma_D\rightarrow \Gamma_N}\left[\phi_0\vert_{\Gamma_D}\right](x)\,,\quad x\in \Gamma_N\,,
		\label{compute1}
	\end{eqnarray}
	where $\phi_0\vert_{\Gamma_N}$ and $\phi_0\vert_{\Gamma_D}$ are given by the solution of
	\begin{eqnarray}
		\left(\begin{array}{cc}
			\mathrm{S}^{\Gamma_N\rightarrow \Gamma_D} & \mathrm{S}^{\Gamma_D\rightarrow \Gamma_D}\\
			\frac{1}{2}I+\mathrm{T}^{\Gamma_N\rightarrow \Gamma_N} &\mathrm{T}^{\Gamma_D\rightarrow \Gamma_N}
		\end{array}\right)
		\left(\begin{array}{cc}
			\phi_0\vert_{\Gamma_N}\\
			\phi_0\vert_{\Gamma_D}
		\end{array}\right)=
		\left(\begin{array}{c}
			0\\
			g
		\end{array}
		\right)\,.
		\label{help1}
	\end{eqnarray}
\end{proposition}
\begin{proof} 
We use a single-layer ansatz to represent the solution $u_0$ inside $D$ as
\begin{eqnarray}
	u_0(x)&=&\mathrm{SL}^{\partial D} \left[\phi_0 \right](x) = \mathrm{SL}^{\Gamma_N} \left[\phi_0\vert_{\Gamma_N}\right](x)+\mathrm{SL}^{\Gamma_D}\left[\phi_0\vert_{\Gamma_D}\right](x)\,,\qquad x\in D\,,
	\label{start}
\end{eqnarray}
where we used the fact that the given boundary is a disjoint union of $\Gamma_N$ and $\Gamma_D$. Here, $\phi_0\vert_{\Gamma_N}$ and $\phi_0\vert_{\Gamma_D}$ are yet unknown functions.
 Letting $D\ni x\rightarrow x\in \Gamma_D$ in (\ref{start}) together with the jump relation and the boundary condition $u_0\vert_{\Gamma_D}=0$ gives the first boundary integral equation
\begin{eqnarray}
	0= \mathrm{S}^{\Gamma_N\rightarrow \Gamma_D}\left[\phi_0\vert_{\Gamma_N}\right](x)+\mathrm{S}^{\Gamma_D\rightarrow \Gamma_D}\left[\phi_0\vert_{\Gamma_D}\right](x)\,,\quad x\in \Gamma_D\,.
	\label{one}
\end{eqnarray}
Taking the normal derivative of (\ref{start}) and letting $D\ni x\rightarrow x\in \Gamma_N$ along with the jump and the boundary condition $\partial_{\nu}u_0\vert_{\Gamma_N}=g$ yields the second boundary integral equation
\begin{eqnarray}
	g(x)=\mathrm{T}^{\Gamma_N\rightarrow \Gamma_N}\left[\phi_0\vert_{\Gamma_N}\right](x)+\frac{1}{2}\phi_0\vert_{\Gamma_N}(x)+\mathrm{T}^{\Gamma_D\rightarrow \Gamma_N}\left[\phi_0\vert_{\Gamma_D}\right](x)\,,\quad x\in \Gamma_N\,.
	\label{two}
\end{eqnarray}
Equations (\ref{one}) and (\ref{two}) can be written together as the system (\ref{help1})
which have to be solved for $\phi_0\vert_{\Gamma_N}$ and $\phi_0\vert_{\Gamma_D}$. Here, $I$ denotes the identity operator. With this, we can use (\ref{start}) to obtain $u_0$ at any point within $D$. Letting $D\ni x\rightarrow x\in \Gamma_N$ along with the jump relations yields (\ref{compute1}).
\end{proof}

\subsection{Integral equation for computing $u$ on $\Gamma_N$}
Next, we consider the corroded object $D\backslash \overline{\Omega}$, refer also to \eqref{dp1} . Now, we are in position to explain how to obtain $u$ at any point of $\Gamma_N$.
\begin{proposition}
	Let $D\backslash \overline{\Omega}$ be the domain representing the corroded object with boundary $\overline{\Gamma_N }\cup \overline{\Gamma_C}$ satisfying $\Gamma_N\cap \Gamma_C=\emptyset$. Then, the solution $u$ to \eqref{dp1} on $\Gamma_N$ for the corroded object is given by
	\begin{align}
		u\vert_{\Gamma_N}(x)=\mathrm{S}^{\Gamma_N\rightarrow \Gamma_N}\left[\phi\vert_{\Gamma_N}\right](x)+\mathrm{S}^{\Gamma_C\rightarrow \Gamma_N}\left[\phi\vert_{\Gamma_C}\right](x)\,,\quad x\in \Gamma_N\,,
		\label{compute2}
	\end{align}
	 where $\phi\vert_{\Gamma_N}$ and $\phi\vert_{\Gamma_C}$ are given by the solution of
	 \begin{align}
	 	\left(\begin{array}{cc}
	 		\mathrm{T}^{\Gamma_N\rightarrow \Gamma_C}+\gamma \mathrm{S}^{\Gamma_N\rightarrow \Gamma_C} & \frac{1}{2}I+\mathrm{T}^{\Gamma_C\rightarrow \Gamma_C}+ \gamma\mathrm{S}^{\Gamma_C\rightarrow \Gamma_C}\\
	 		\frac{1}{2}I+\mathrm{T}^{\Gamma_N\rightarrow \Gamma_N} &\mathrm{T}^{\Gamma_C\rightarrow \Gamma_N}
	 	\end{array}\right)
	 	\left(\begin{array}{cc}
	 		\phi\vert_{\Gamma_N}\\
	 		\phi\vert_{\Gamma_C}
	 	\end{array}\right)=
	 	\left(\begin{array}{c}
	 		0\\
	 		g
	 	\end{array}
	 	\right)\,.
	 	\label{help2}
	 \end{align}
\end{proposition}

\begin{proof}
Using a single-layer ansatz 
\begin{eqnarray}
	u(x)&=&\mathrm{SL}^{\partial (D\backslash \overline{\Omega})} \left[\phi\right](x) = \mathrm{SL}^{\Gamma_N} \left[\phi\vert_{\Gamma_N}\right](x)+\mathrm{SL}^{\Gamma_C}\left[\phi\vert_{\Gamma_C}\right](x)\,,\qquad x\in D\backslash \overline{\Omega}\,,
	\label{start2}
\end{eqnarray}
where $\phi\vert_{\Gamma_N}$ and $\phi\vert_{\Gamma_C}$ are again unknown functions. As before, we obtain on $\Gamma_C$
\begin{eqnarray*}
	u\vert_{\Gamma_C}(x)&=&\mathrm{S}^{\Gamma_N\rightarrow \Gamma_C} \left[\phi\vert_{\Gamma_N}\right](x)+\mathrm{S}^{\Gamma_C\rightarrow\Gamma_C}\left[\phi\vert_{\Gamma_C}\right](x)\,,\qquad x\in \Gamma_C\,,\\
	\partial_\nu u\vert_{\Gamma_C}(x)&=&\mathrm{T}^{\Gamma_N\rightarrow \Gamma_C} \left[\phi\vert_{\Gamma_N}\right](x)+\mathrm{T}^{\Gamma_C\rightarrow\Gamma_C}\left[\phi\vert_{\Gamma_C}\right](x)+\frac{1}{2}\phi\vert_{\Gamma_C}\,,\qquad x\in \Gamma_C\,,
\end{eqnarray*}
and hence using the boundary condition $\partial_\nu u+\gamma u=0$ on $\Gamma_C$ we obtain the first boundary integral equation
\begin{eqnarray}
	0&=&\mathrm{T}^{\Gamma_N\rightarrow \Gamma_C}\left[\phi\vert_{\Gamma_N}\right](x)+\mathrm{T}^{\Gamma_C\rightarrow\Gamma_C}\left[\phi\vert_{\Gamma_C}\right](x)+\frac{1}{2}\phi\vert_{\Gamma_C}\nonumber\\
	&+&\gamma \mathrm{S}^{\Gamma_N\rightarrow \Gamma_C} \left[\phi\vert_{\Gamma_N}\right](x)+\gamma\mathrm{S}^{\Gamma_C\rightarrow\Gamma_C}\left[\phi\vert_{\Gamma_C}\right](x)\,,\qquad x\in \Gamma_C\,.
	\label{one2}
\end{eqnarray}
Using the boundary condition $\partial_\nu u\vert_{\Gamma_N}=g$ yields the second boundary integral equation 
\begin{eqnarray}
	g(x)=\mathrm{T}^{\Gamma_N\rightarrow \Gamma_N}\left[\phi\vert_{\Gamma_N}\right](x)+\frac{1}{2}\phi\vert_{\Gamma_N(x)}+\mathrm{T}^{\Gamma_C\rightarrow \Gamma_N}\left[\phi\vert_{\Gamma_C}\right](x)\,,\quad x\in \Gamma_N\,.
	\label{two2}
\end{eqnarray}
Equations (\ref{one2}) and (\ref{two2}) can be written together as the system (\ref{help2})
which have to be solved for $\phi\vert_{\Gamma_N}$ and $\phi\vert_{\Gamma_C}$. With this, we can use (\ref{start2}) to obtain $u$ at any point within $D\backslash \overline{\Omega}$. Letting $D\backslash \overline{\Omega}\ni x\rightarrow x\in \Gamma_N$ along with the jump condition yields (\ref{compute2}).
\end{proof}

\subsection{Integral equation for computing $\mathbb{G}(\cdotp,z)$}
In order to solve the inverse shape problem, for fixed $z\in D$, we need to compute $\mathbb{G}(\cdotp,z)$ on $\Gamma_N$ (refer also to \eqref{volt-gap}). Recall, that $\mathbb{G}(\cdotp,z)$ satisfies
\begin{eqnarray*}
-\Delta \mathbb{G}(\cdot, z) = \delta(\cdot -z) \textrm{ in } D,  \quad \dnu{\mathbb{G}(\cdot, z)} = 0 \textrm{ on } \Gamma_N, \quad \text{ and } \quad \mathbb{G}(\cdot, z) = 0 \textrm{ on } \, {\Gamma_D}.
\end{eqnarray*}
Just as in \cite{Zaremba}, we assume that 
\[\mathbb{G}(\cdotp,z)=w(\cdotp,z)+\Phi(\cdotp,z)\,,\] 
where $\Phi(\cdotp,z)$ is again the fundamental solution of the Laplace equation in $\mathbb{R}^2$. Then, $w(\cdotp,z)$ obviously satisfies
\begin{eqnarray*}
\Delta w(\cdotp,z)= 0   \textrm{ in }  D,   \quad  \partial_{\nu}w(\cdotp,z)=-\partial_\nu\Phi(\cdotp,z)  \textrm{ on }   \Gamma_N ,  \quad  w(\cdotp,z)=-\Phi(\cdotp,z)  \textrm{ on }   \Gamma_D\,.
\end{eqnarray*}
Our task now is to compute $w(\cdotp,z)$ on $\Gamma_N$ in order to approximate $\mathbb{G}(\cdotp,z)$ on $\Gamma_N$.
\begin{proposition}
	Let $z\in D$ be fixed. Then $\mathbb{G}(\cdotp,z)$ on $\Gamma_N$ is given by
	\[\mathbb{G}(\cdotp,z)\vert_{\Gamma_N}=w(\cdotp,z)\vert_{\Gamma_N}+\Phi(\cdotp,z)\vert_{\Gamma_N}\,.\]
	Here, $w(\cdotp,z)\vert_{\Gamma_N}$ is obtained through
	\begin{eqnarray}
		w(x,z)\vert_{\Gamma_N}=\mathrm{S}^{\Gamma_N\rightarrow \Gamma_N}\left[\phi_z\vert_{\Gamma_N}\right](x)+\mathrm{S}^{\Gamma_D\rightarrow \Gamma_N}\left[\phi_z\vert_{\Gamma_D}\right](x)\,,\quad x\in \Gamma_N\,,
		\label{compute3}
	\end{eqnarray}
	where $\phi_z\vert_{\Gamma_N}$ and $\phi_z\vert_{\Gamma_D}$ are given by the solution of
	 \begin{eqnarray}
	 	\left(\begin{array}{cc}
	 		\mathrm{S}^{\Gamma_N\rightarrow \Gamma_D} & \mathrm{S}^{\Gamma_D\rightarrow \Gamma_D}\\
	 		\frac{1}{2}I+\mathrm{T}^{\Gamma_N\rightarrow \Gamma_N} &\mathrm{T}^{\Gamma_D\rightarrow \Gamma_N}
	 	\end{array}\right)
	 	\left(\begin{array}{cc}
	 		\phi_z \vert_{\Gamma_N}\\
	 		\phi_z \vert_{\Gamma_D}
	 	\end{array}\right)=
	 	\left(\begin{array}{c}
	 		-\Phi(\cdotp,z) \vert_{\Gamma_N} \\
	 		-\partial_\nu\Phi(\cdotp,z) \vert_{\Gamma_D}
	 	\end{array}
	 	\right)\,.
	 	\label{help3}
	 \end{eqnarray}
\end{proposition}
\begin{proof}
We make the ansatz
\begin{eqnarray}
	w(x,z)&=&\mathrm{SL}^{\partial D} \left[\phi_z\right](x) =\mathrm{SL}^{\Gamma_N} \left[\phi_z\vert_{\Gamma_N}\right](x)+\mathrm{SL}^{\Gamma_D}\left[\phi_z\vert_{\Gamma_D}\right](x)\,,\quad x\in D\,.
	\label{start3}
\end{eqnarray}
Here, $\phi_z\vert_{\Gamma_N}$ and $\phi_z\vert_{\Gamma_D}$ are yet unknown functions.
Letting $D\ni x\rightarrow x\in \Gamma_D$ in (\ref{start3}) together with the jump conditions and the boundary condition $w(\cdotp,z)\vert_{\Gamma_D}=-\Phi(\cdotp,z)\vert_{\Gamma_D}$ gives the first boundary integral equation
\begin{eqnarray}
	-\Phi(x,z)= \mathrm{S}^{\Gamma_N\rightarrow \Gamma_D}\left[\phi_z\vert_{\Gamma_N}\right](x)+\mathrm{S}^{\Gamma_D\rightarrow \Gamma_D}\left[\phi_z\vert_{\Gamma_D}\right](x)\,,\quad x\in \Gamma_D\,.
	\label{one3}
\end{eqnarray}
Taking the normal derivative of (\ref{start3}) and letting $D\ni x\rightarrow x\in \Gamma_N$ along with the jump condition and the boundary condition $\partial_{\nu}w(\cdotp,z)\vert_{\Gamma_N}=-\partial_\nu\Phi(\cdotp,z)\vert_{\Gamma_N}$ yields the second boundary integral equation
\begin{align}
	-\partial_\nu\Phi(x,z)=\mathrm{T}^{\Gamma_N\rightarrow \Gamma_N}\left[\phi_z\vert_{\Gamma_N}\right](x)+\frac{1}{2}\phi_z\vert_{\Gamma_N}(x)+\mathrm{T}^{\Gamma_D\rightarrow \Gamma_N}\left[\phi_z\vert_{\Gamma_D}\right](x)\,,\,\, x\in \Gamma_N\,.
	\label{two3}
\end{align}
Equations (\ref{one3}) and (\ref{two3}) can be written together as the system (\ref{help3}) which we have to solve for $\phi_z\vert_{\Gamma_N}$ and $\phi_z\vert_{\Gamma_D}$. With this, we can use (\ref{start3}) to obtain $w(\cdotp,z)$ at any point within $D$. Letting $D\ni x\rightarrow x\in \Gamma_N$ along with the jump conditions yields (\ref{compute3}).
\end{proof}

\subsection{Discretization of the DtN operator $\Lambda-\Lambda_0$}
Now, we illustrate how to approximate the equation 
\begin{eqnarray}
 (\Lambda-\Lambda_0)g_z=\mathbb{G}(\cdotp,z)\vert_{\Gamma_N}
 \label{myDtN}
\end{eqnarray}
with the Galerkin method for a fixed $z\in D$. Without loss of generality, we assume that the functions on $\Gamma_N$ are parametrized by $x(\theta)$ such that $\theta \in (\theta_1,\theta_2)\subseteq [0,2\pi]$. For a given set of `Fourier' basis functions $\varphi_n(\theta)$ and yet unknown `Fourier' coefficients $g_n^{(z)}$ 
\[g_z(\theta)=\sum_{n=0}^{\infty} g_n^{(z)}\cdotp \varphi_n(\theta)\]
which is approximated by a finite sum
\[g_z(\theta) \approx \sum_{n=0}^{N_{{B}}} g_n^{(z)}\cdotp \varphi_n(\theta)\]
(i.e. $N_B +1$ denotes the number of basis functions). Now, we insert this into (\ref{myDtN}) to obtain
\[\sum_{n=0}^{N_{\text{B}}}g_n^{(z)}\cdotp (\Lambda-\Lambda_0)\varphi_n(\theta)=\mathbb{G}(\theta,z)\,.\]
Multiplying this equation with $\varphi_m (\theta)$ for $m\in \{0,1,\ldots,N_{B}\}$ and integrating over $[\theta_1,\theta_2]$ yields the linear system of size $(N_{{B}}+1)\times (N_{{B}}+1)$
\[\sum_{n=0}^{N_B}g_n^{(z)} \int_{\theta_1}^{\theta_2} \varphi_m(\theta)\cdotp  (\Lambda-\Lambda_0)\varphi_n(\theta)\,\mathrm{d}s(\theta) =\int_{\theta_1}^{\theta_2}\varphi_m(\theta)\cdotp \mathbb{G}(\theta,z)\,\mathrm{d}s(\theta) \,,\]
where the unknown `Fourier' coefficients $g_n^{(z)}$ are to be determined. We write this linear system abstractly as 
\[Bg^{(z)}=b^{(z)}\,.\]
To compute the matrix entries for each $m$ and $n$ numerically
\begin{eqnarray}
B_{mn}=\int_{\theta_1}^{\theta_2} \varphi_m(\theta)\cdotp  (\Lambda-\Lambda_0)\varphi_n(\theta)\,\mathrm{d}s(\theta) 
\label{matrix}
\end{eqnarray}
we subdivide the interval $[\theta_1,\theta_2]$ into $n_f$ equidistant panels and apply to each panel Gau\ss{}-Legendre quadrature using three quadrature nodes and three weights. 
 Note that the matrix $B$ should become symmetric for increasing $n_f$, since the operator $(\Lambda-\Lambda_0)$ is self-adjoint. In the same way, we approximate the right hand side for each $m$
\begin{eqnarray}b^{(z)}_m=\int_{\theta_1}^{\theta_2}\varphi_m(\theta)\cdotp \mathbb{G}(\theta,z)\,\mathrm{d}s(\theta) \,.
\label{rhs}
\end{eqnarray}
\begin{remark}
To obtain $(u-u_0 )\vert_{\Gamma_N} = (\Lambda-\Lambda_0)\varphi_n(\theta)$ as well as $\mathbb{G}(\theta,z)$ for fixed $z\in D$ (compare also (\ref{compute1}) and (\ref{compute2}) as well as (\ref{compute3}) for the corresponding integral equation), we use as discretization the boundary element collocation method as done in \cite{kleefeldhotspot} using the wave number $0$. We use $\alpha=(1-\sqrt{3/5})/2$, then the collocation nodes on $\Gamma_N$ are exactly the three Gau\ss{}-Legendre nodes on each panel that are needed in the approximation of (\ref{matrix}) and (\ref{rhs}), respectively. Hence, we are now in position to create the matrix $B$ which approximates $(\Lambda-\Lambda_0)$ and the right hand side $b^{(z)}$ for different domains. That is, we can now create synthetic data which then can be used for the reconstruction algorithms LSM or FM.
\end{remark}

\subsection{Reconstructions with FMreg and LSMreg}
Now, we present some reconstructions using FM and LSM. For the details on the implementation of the FM, we refer the reader to \cite{kirschbook}. We compute $W_{\mathrm{FM}}(z)$ (see section \ref{FMsect} for the definition) for $z \in \mathcal{G}$, where $\mathcal{G}$ is a set of grid points covering the domain of interest. In the following, we plot ${\displaystyle W_{\mathrm{FM}}^{\log}(z)=\log (W_{\mathrm{FM}}(z))}$. With this definition, we have that for ${\displaystyle z\in \Omega \iff W_{\mathrm{FM}}^{\log}(z)\gg 0}$ and ${\displaystyle z\in D\backslash\Omega \iff W_{\mathrm{FM}}^{\log}(z)\ll 0}$. We will use regularization for the factorization method (FMreg) by only using the singular values that are greater than $10^{-5}$ since the problem is severely ill-posed. We denote the regularized version of $W_{\mathrm{FM}}^{\log}(z)$ by ${\displaystyle W_{\mathrm{FMreg}}^{\log}(z)}$. The regularized version of the linear sampling method (LSMreg) which is denoted by ${\displaystyle W_{\mathrm{LSMreg}}^{\log}(z)}$ (refer to section \ref{LSMsect} for details), where we solve 
$$(B^\ast B+\alpha I)g^{(z)}=B^\ast b^{(z)} \quad \text{with} \quad \alpha=10^{-5}$$ 
the Tikhonov regularization of $Bg^{(z)}=b^{(z)}$.\\

\noindent{\bf Example 1:} 
Let the domain be a square $D=[0,2\pi]\times [-2\pi,0]$ completely buried and only the upper part of the domain is visible. Parts of the square are corroded as shown in Figure \ref{figex2new}. Precisely, the corroded part is given by the polygon with vertices $(2\pi,0)$, $(0,0)$, $(\pi/2,-3\pi/2)$, and $(3\pi/2,-3\pi/2)$. 
 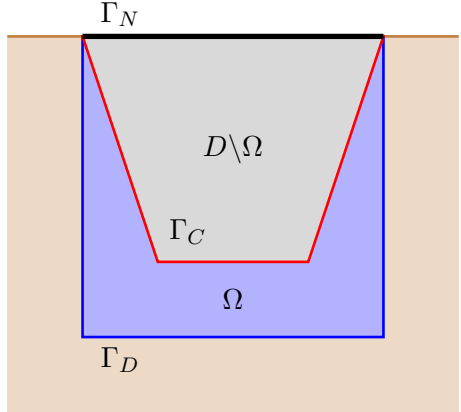
\begin{figure}[H]
		\centering
		\begin{tikzpicture}
			\draw [fill=white!30, draw=none] (-3,-3) rectangle (3,3);
			\draw [fill=brown!30, draw=none] (-3,-3) rectangle (3,2);
			\draw[line width=1pt, brown] (-3,2) -- (3,2);
			\draw [line width=1pt,fill=blue!30, draw=blue] (-2,-2) rectangle (2,2);			
			\draw[line width=1pt, fill=gray!30, draw=red] (2,2) -- (-2,2) -- (-1,-1) -- (1,-1) -- cycle;
			\draw[line width=2pt, black] (-2,2) -- (2,2);			
			\node at (0,0.5) {$D\backslash \Omega$};
			\node at (0,-1.5) {$\Omega$};
			\node at (-0.6,-0.6) {$\Gamma_C$};
			\node at (-1.5,-2.3) {$\Gamma_D$};
			\node at (-1.5,2.3) {$\Gamma_N$};
		\end{tikzpicture}
		\caption{\label{figex2new}The buried object for Example 1.}
	\end{figure}
To create the data, we use the 20 basis functions $\varphi_n(\theta)=\cos(n\theta)$ for $n\in \{0,1,\ldots,19\}$ and $\theta\in [0,2\pi]$ and $n_f=300$ on $\Gamma_N$ for the boundary element collocation method. Furthermore, for simplicity we use $\gamma=1/2$ and $\gamma=2$. For the FMreg and LSMreg, we use an equidistant grid of size $100\times 100$ of $D$. We choose the level--curve$=3/2$ as threshold value for recovering $\Gamma_C$ by the FMreg imaging functional. For the LSMreg imaging functional we choose the level--curve$=-1/2$ as threshold value. In Figure \ref{figex2new2} and \ref{figex2new21}, we present the reconstruction results with the FMreg and the LSMreg.
	\begin{figure}[H]
		\centering
		\includegraphics[width=6.5cm]{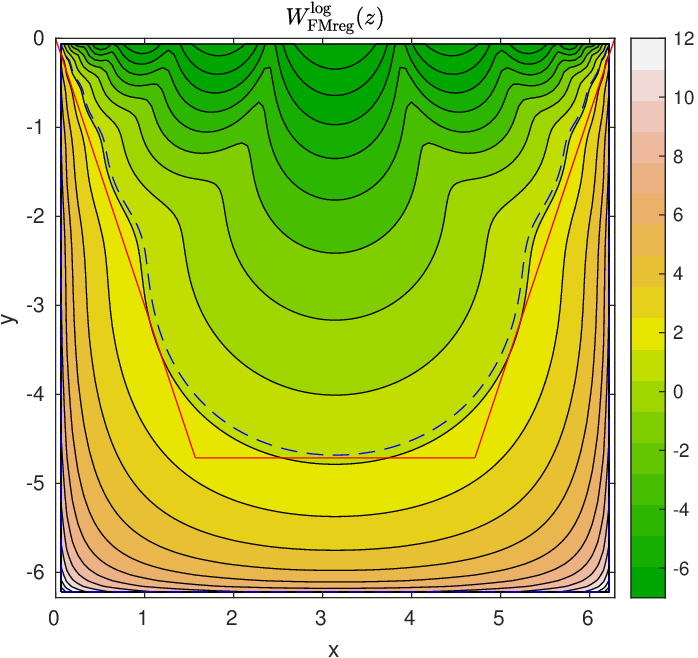}
		\includegraphics[width=6.5cm]{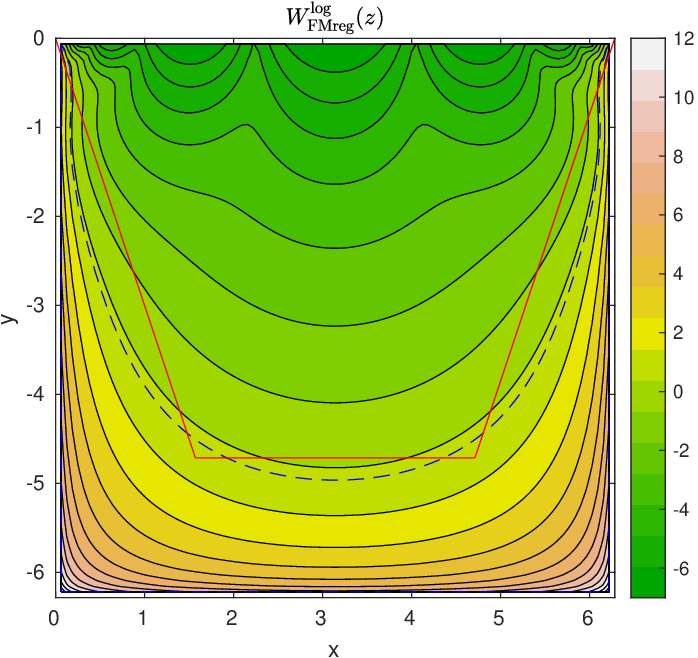}
		\caption{\label{figex2new2}Reconstructions via the FMreg imaging functional. The dashed blue line is the approximated boundary of the reconstruction. Left: $\gamma=1/2$ and Right: $\gamma=2$.}
	\end{figure}
	
		\begin{figure}[H]
		\centering
		\includegraphics[width=6.5cm]{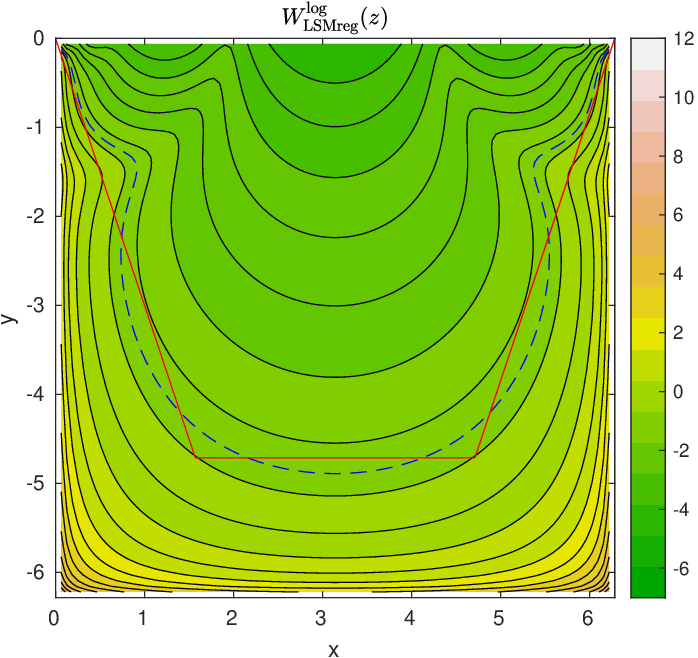}
		\includegraphics[width=6.5cm]{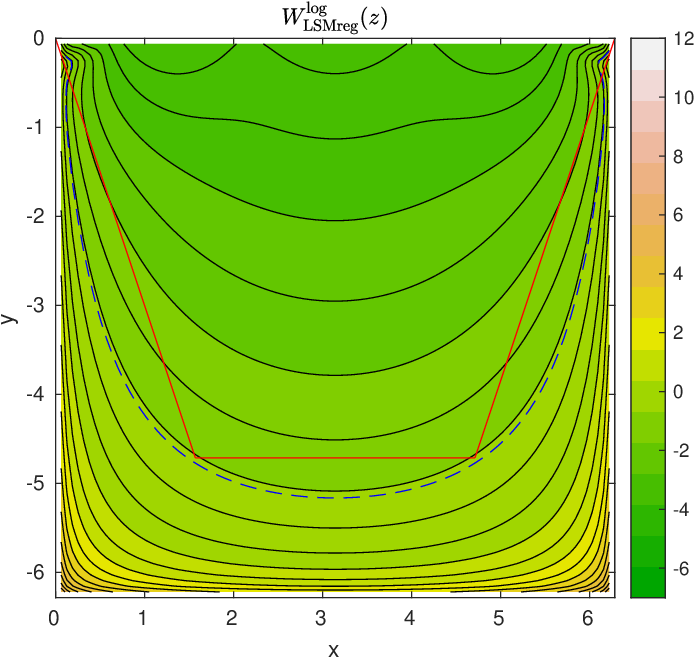}
		\caption{\label{figex2new21}Reconstructions via the LSMreg imaging functional. The dashed blue line is the approximated boundary of the reconstruction. Left: $\gamma=1/2$ and $\gamma=2$ and Right: $\gamma=2$.}
	\end{figure}
We observe reasonable reconstructions using the FMreg and LSMreg although not perfect which is expected as the problem is severely ill-posed. \\

\noindent{\bf Example 2:} We now consider a wedge-shaped domain using the angle of $\pi/2$. A certain part of it is corroded as shown in Figure \ref{figex3new}. Precisely, the corroded part is given by the triangle with vertices $(1,0)$, $(0,1)$, and $(0,0)$.
 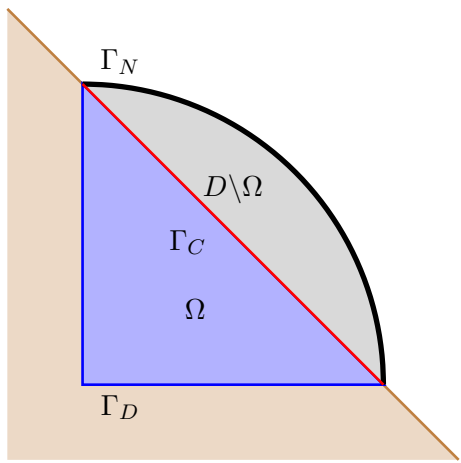
\begin{figure}[H]
		\centering
		\begin{tikzpicture}
		   	 \draw [fill=white!30, draw=none] (-3,-3) rectangle (3,3);
		   	 \draw[line width=1pt, fill=brown!30, draw=none] (-3,3) -- (3,-3) -- (-3,-3) -- cycle;
			\draw[line width=1pt, brown] (-3,3) -- (3,-3);
			\draw[line width=1pt, fill=blue!30, draw=blue] (-2,2) -- (2,-2) -- (-2,-2) -- cycle;
			\draw[line width=2pt, domain=0:pi/2, smooth, variable=\x, black, fill=gray!30, scale=2] plot ({2*cos(\x r)-1}, {2*sin(\x r)-1});
			\draw[line width=1pt, red] (-2,2) -- (2,-2);
			\node at (0,0.6) {$D\backslash \Omega$};
			\node at (-0.5,-1.0) {$\Omega$};
			\node at (-0.6,-0.1) {$\Gamma_C$};
			\node at (-1.5,-2.3) {$\Gamma_D$};
			\node at (-1.5,2.3) {$\Gamma_N$};
		\end{tikzpicture}
		\caption{\label{figex3new}The buried object for Example 2.}
	\end{figure}
We use as 20 basis functions $\varphi_n(\theta)=\cos(4n\theta)$ for $n\in \{0,1,\ldots,19\}$ with $\theta\in [0,\pi/2]$ and $n_f=300$ on $\Gamma_N$ for the boundary element collocation method. Further, we use $\gamma=1/2$ and $\gamma=2$. For the FMreg and LSMreg, we use an equidistant grid of size $100\times 100$ of $[0,1]\times [0,1]$. We choose the level--curve$=0.25$  and $=-1$ as threshold value for the FMreg imaging functional when $\gamma=1/2$ and $\gamma=2$, respectively. For the LSMreg imaging functional we choose level--curve$=-1$ and $=-1.5$ as threshold value for $\gamma=1/2$ and $\gamma=2$. In Figure \ref{figex3new2} and \ref{figex3new21}, we present the reconstruction results with the FMreg and the LSMreg.
	\begin{figure}[H]
		\centering
		\includegraphics[width=6.5cm]{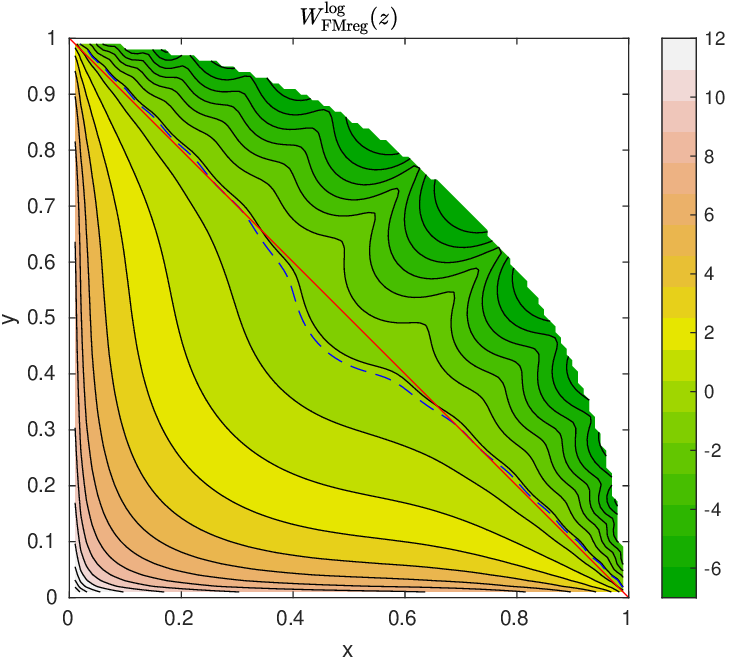}
		\includegraphics[width=6.5cm]{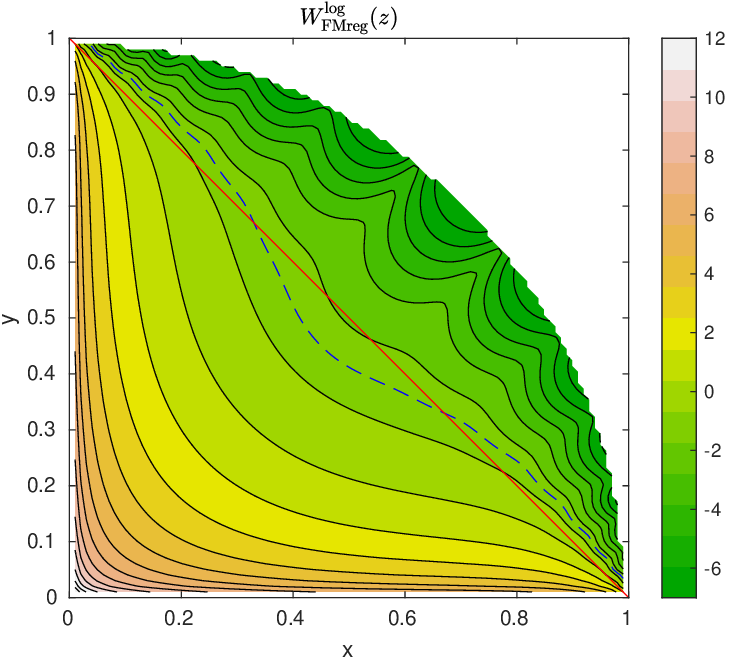}
		\caption{\label{figex3new2}Reconstructions via the FMreg imaging functional. The dashed blue line is the approximated boundary of the reconstruction. Left: $\gamma=1/2$ and Right: $\gamma=2$.}
	\end{figure}
	
		\begin{figure}[H]
		\centering
		\includegraphics[width=6.5cm]{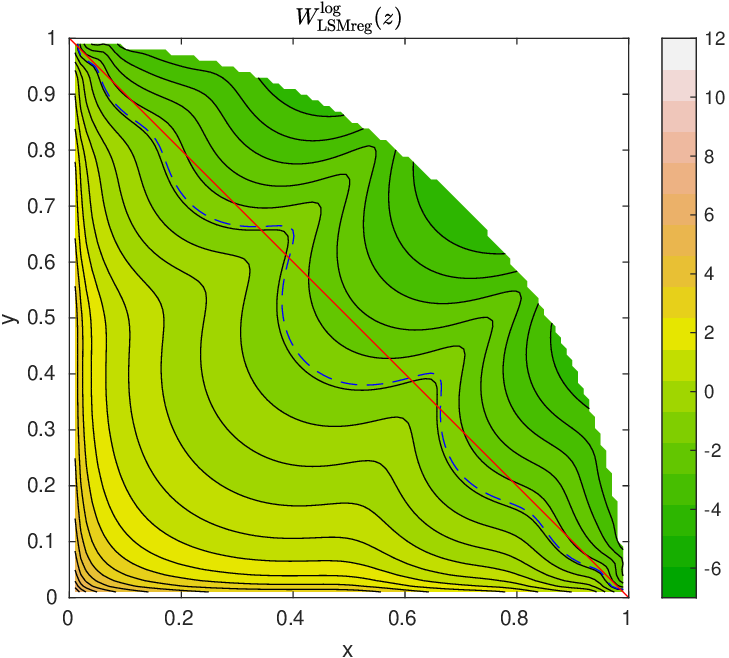}
		\includegraphics[width=6.5cm]{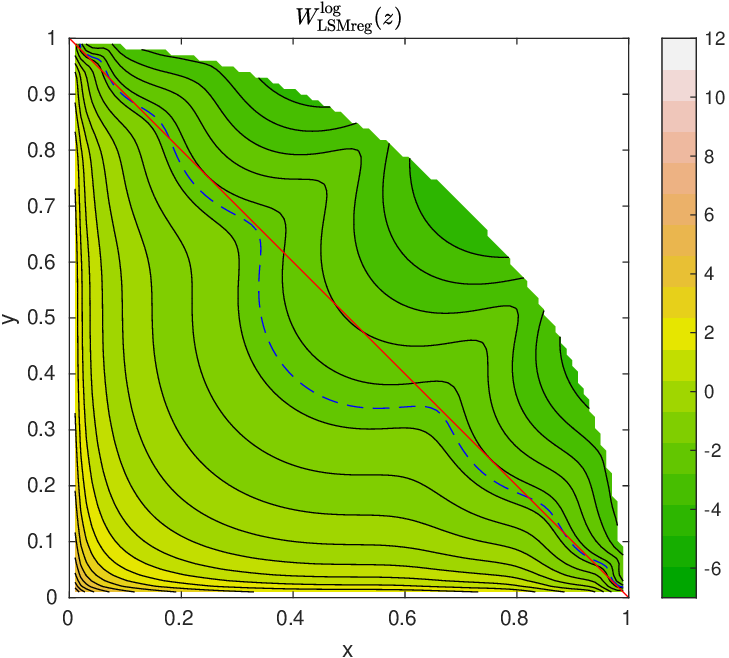}
		\caption{\label{figex3new21}Reconstructions via the LSMreg imaging functional. The dashed blue line is the approximated boundary of the reconstruction. Left: $\gamma=1/2$ and Right: $\gamma=2$.}
	\end{figure}
We observe reasonable reconstructions using the FMreg and LSMreg which is much better than the one for the previous example. The FMreg performs much better. \\

\noindent {\bf Example 3:} Next, we use an ellipse with half-axis $a=1.1$ and $b=1$ that is half buried. Hence, $\Gamma_N$ and $\Gamma_D$ are given in parameter form as $x_1 (\theta)=1.1\cdotp \cos(\theta)$ and $x_2 (\theta)=\sin(\theta)$ with $\theta\in [0,\pi]$ and $\theta\in (\pi,2\pi)$, respectively. The boundary $\Gamma_C$ is given by $x_1 (\theta)=1.1\cdotp\cos(\theta)$ and $x_2(\theta) =0.5\cdotp \sin(\theta)$ with $\theta\in (\pi,2\pi)$. The situation is depicted in Figure \ref{figex4new}. 
 \begin{figure}[H]
		\centering
		\begin{tikzpicture}
		    	\draw [fill=white!30, draw=none] (-3,-3) rectangle (3,3);
			\draw [fill=brown!30, draw=none] (-3,-3) rectangle (3,0);
			\draw[line width=1pt, brown] (-3,0) -- (3,0);
			\draw[line width=1pt, domain=0:2*pi, smooth, variable=\x, black, fill=gray!30, scale=2] plot ({1.1*cos(\x r)}, {sin(\x r)});
			\draw[line width=1pt, domain=pi+0.01:2*pi-0.01, smooth, variable=\x, blue, fill=blue!30, scale=2] plot ({1.1*cos(\x r)}, {sin(\x r)});
			\draw[line width=1pt, domain=pi:2*pi, smooth, variable=\x, blue,fill=gray!30, scale=2] plot ({1.1*cos(\x r)}, {0.5*sin(\x r)});
			\draw[line width=1pt, domain=pi:2*pi, smooth, variable=\x, red,fill=none, scale=2] plot ({1.1*cos(\x r)}, {0.5*sin(\x r)});
			\draw[line width=2pt, domain=0:pi, smooth, variable=\x, black, fill=none, scale=2] plot ({1.1*cos(\x r)}, {sin(\x r)});
			\node at (0,0.6) {$D\backslash \overline{\Omega}$};
			\node at (0,-1.7) {$\Omega$};
			\node at (-1.5,-1.9) {$\Gamma_D$};
			\node at (-1.5,1.9) {$\Gamma_N$};
			\node at (-1.0,-0.6) {$\Gamma_C$};
		\end{tikzpicture}
		\caption{\label{figex4new}The buried object for Example 3.}
	\end{figure}
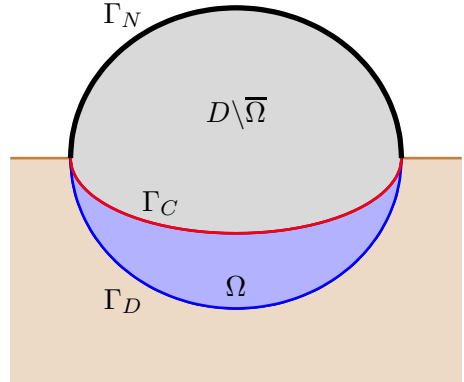
We use as 20 basis functions $\varphi_n(\theta)=\cos(2n\theta)$ for $n\in \{0,1,\ldots,19\}$ with $\theta\in [0,\pi/2]$ and $n_f=300$ on $\Gamma_N$ for the boundary element collocation method. For this example, we use an equidistant grid of size $100\times 100$ of $[-1.1,1.1]\times [-1.1,1.1]$. To reconstruct $\Gamma_C$ we choose the level--curve$=2.5$  and $=1.5$ as threshold value for the FMreg imaging functional when $\gamma=1/2$ and $\gamma=2$, respectively. In Figure \ref{figex4new2}, we present the reconstruction results with the FMreg.
	\begin{figure}[H]
		\centering
		\includegraphics[width=6.5cm]{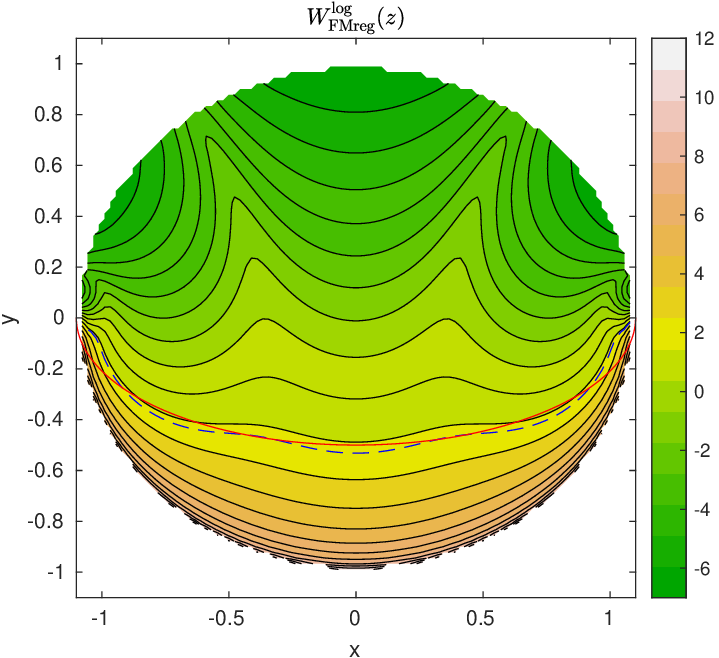}
		\includegraphics[width=6.5cm]{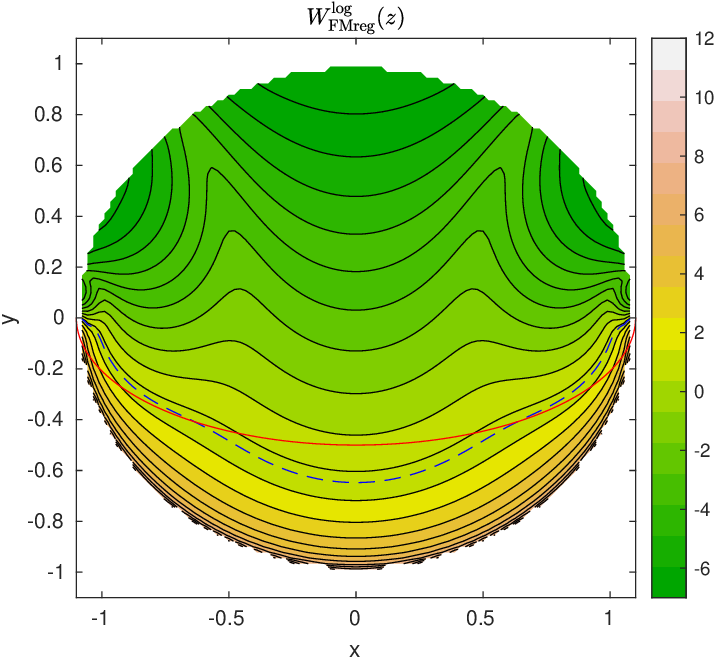}
		\caption{\label{figex4new2}Reconstructions via the FMreg imaging functional. The dashed blue line is the approximated boundary of the reconstruction. Left: FMreg for $\gamma=1/2$ and Right: $\gamma=2$.}
	\end{figure}
We observe good reconstructions using the FMreg which is much better than the one for the previous two examples.  

\section{Summary and Conclusion}
We have extended the applicability of the LSM and FM for recovering an unknown corroded boundary from partial Cauchy data. This main idea that we employed was to embed the `defective' region into a `healthy' region. This allows use to consider the problem of finding the corroded region. We also want to note that the analysis presented here also implies that the generalized linear sampling method \cite{GLSM} is a valid reconstruction method. In the numerical results, we have seen that the threshold value seems to depend on $\gamma$. To obtain the correct dependence, a further investigation is needed. Moreover, the FMreg seems to provide better reconstructions than the LSMreg which might be due to the choice of the regularization parameter. A thorough investigation is needed to find out if the correct choice of the regularization parameter on both methods can give similar reconstruction results. In sum, we are able to obtain reasonable reconstructions with both the FMreg and the LSMreg giving us a good idea of how much of the buried obstacle is corroded although the problem at hand is severely ill-posed. \\

\noindent{\bf Acknowledgments:} The research of I. Harris and H. Lee is partially supported by the NSF DMS Grant 2107891. 



\end{document}